\documentclass[letterpaper, 11pt,  reqno]{amsart}

\usepackage[margin=1.1in,marginparwidth=1.5cm, marginparsep=0.5cm]{geometry}

\usepackage{amsmath,amssymb,amscd,amsthm,amsxtra, esint}

\usepackage{mathrsfs} 

\usepackage[implicit=true]{hyperref}

\usepackage{color} 
\usepackage{ulem}
\usepackage[makeroom]{cancel}

\allowdisplaybreaks[2]

\usepackage{color}

\definecolor{gr}{rgb}   {0.,   0.69,   0.23 }
\definecolor{bl}{rgb}   {0.,   0.5,   1. }
\definecolor{mg}{rgb}   {0.85,  0.,    0.85}
\definecolor{yl}{rgb}   {0.8,  0.7,   0.}
\definecolor{or}{rgb}  {0.7,0.2,0.2}

\newtheorem{theorem}{Theorem} [section]

\newtheorem{lemma}[theorem]{Lemma}

\newtheorem{remark}[theorem]{Remark}

\DeclareMathOperator*{\supp}{supp}

\newcommand{\noi}{\noindent}
\newcommand{\Z}{\mathbb{Z}}
\newcommand{\R}{\mathbb{R}}

\newcommand{\T}{\mathbb{T}}

\let\Re=\undefined\DeclareMathOperator*{\Re}{Re}
\let\Im=\undefined\DeclareMathOperator*{\Im}{Im}

\let\P= \undefined
\newcommand{\P}{\mathbf{P}}

\newcommand{\E}{\mathbb{E}}

\newcommand{\F}{\mathcal{F}}

\newcommand{\al}{\alpha}
\newcommand{\be}{\beta}

\newcommand{\nb}{\nabla}

\newcommand{\Dl}{\Delta}
\newcommand{\eps}{\varepsilon}

\newcommand{\g}{\gamma}

\newcommand{\s}{\sigma}

\newcommand{\ft}{\widehat}

\newcommand{\wt}{\widetilde}

\newcommand{\dx}{\partial_x}
\newcommand{\dt}{\partial_t}
\newcommand{\dd}{\partial}

\newcommand{\ta}{\theta}

\renewcommand{\O}{\Omega}

\newcommand{\les}{\lesssim}

\newcommand{\jb}[1]
{\langle #1 \rangle}

\newcommand{\ind}{\mathbf 1}

\renewcommand{\S}{\mathcal{S}}

\newcommand{\N}{\mathbb{N}}

\newcommand{\EE}{\mathcal{E}}

\numberwithin{equation}{section}
\numberwithin{theorem}{section}

\newcommand{\ze}{\zeta}

\newcommand{\PP}{\mathbb{P}}

\newcommand{\CC}{\mathcal{C}}

\DeclareMathOperator{\Law}{Law}

\newcommand{\V}{\mathcal{V}}

\newcommand{\dr}{\theta}
\newcommand{\Dr}{\Theta}

\newcommand{\Ha}{\mathbb{H}_a}

\renewcommand{\u}{\vec{u}}

\newcommand{\GNS}{\textup{GNS}}
\newcommand{\per}{\textup{per}}

\makeatletter
\@namedef{subjclassname@2020}{%
	\textup{2020} Mathematics Subject Classification}
\makeatother

\begin{document}
\baselineskip = 14pt

\title[Optimal divergence rate of focusing Gibbs measures]
{Optimal divergence rate of the focusing Gibbs measures}

\author[D.~Greco, G.~Li, R.~Liang, T.~Oh, and Y.~Wang]
{Damiano Greco, Guopeng Li, Rui Liang, Tadahiro Oh, and Yuzhao Wang}

\address{
Damiano Greco\\School of Mathematics\\
The University of Edinburgh\\
and The Maxwell Institute for the Mathematical Sciences\\
James Clerk Maxwell Building\\
The King's Buildings\\
Peter Guthrie Tait Road\\
Edinburgh\\ 
EH9 3FD\\
 United Kingdom}

\email{dgreco@ed.ac.uk}

\address{Guopeng Li, 
School of Mathematics and Statistics, Beijing Institute of Technology, Beijing 100081, China\\
and
School of Mathematics\\
The University of Edinburgh\\
and The Maxwell Institute for the Mathematical Sciences\\
James Clerk Maxwell Building\\
The King's Buildings\\
Peter Guthrie Tait Road\\
Edinburgh\\
EH9 3FD\\
United Kingdom}

\email{guopeng.li@bit.edu.cn}

\address{
Rui Liang\\
School of Mathematical Sciences\\ 
South China Normal University\\ 
Guangzhou\\ 510631, P. R. China, 
and
School of Mathematics\\
University of Birmingham\\ 
Watson Building\\ 
Edgbaston\\
Birmingham\\
B15 2TT\\
United Kingdom, 
and 
Department of Mathematics and Statistics\\
Lederle Graduate Research Tower\\
University of Massachusetts Amherst\\
710 N. Pleasant Street\\
Amherst\\
MA 01003-9305\\
USA 
}
	
\email{ruiliang@umass.edu}

%

\address{
Tadahiro Oh, School of Mathematics\\
The University of Edinburgh\\
and The Maxwell Institute for the Mathematical Sciences\\
James Clerk Maxwell Building\\
The King's Buildings\\
Peter Guthrie Tait Road\\
Edinburgh\\ 
EH9 3FD\\
 United Kingdom, 
 and 
 School of Mathematics and Statistics, Beijing Institute of Technology,
Beijing 100081, China
}

\email{hiro.oh@ed.ac.uk}

\address{
Yuzhao Wang\\
School of Mathematics\\
University of Birmingham\\ 
Watson Building\\ 
Edgbaston\\
Birmingham\\
B15 2TT\\
United Kingdom}
	
	\email{y.wang.14@bham.ac.uk}



\begin{abstract}
We study Gibbs measures on the $d$-dimensional torus with 
 $L^2$-(super)critical 
 focusing 
 interaction potentials.
We establish
a precise divergence rate
of the partition function
as we remove regularization, 
where the optimal constant is given by 
(i)~(the negative of) the minimum value
of the Hamiltonian given an $L^2$-constraint in the $L^2$-critical case
and 
(ii)~the optimal constant for certain Bernstein's inequality
in the mass-supercritical case.
In particular, 
our result in the  $L^2$-critical case precisely quantifies
the  phase transition of the focusing Gibbs measure
at the critical  $L^2$ threshold, 
previously studied by Lebowitz, Rose, and Speer (1988)
and Sosoe, Tolomeo, and the fourth author (2022).

\end{abstract}


\maketitle

\tableofcontents

\newpage

\section{Introduction}
\subsection{Focusing Gibbs measures}
\label{SUBSEC:1.1}

In this paper, we study the  Gibbs measure  $\rho = \rho_{d, s, p}$
on  the $d$-dimensional torus $\mathbb{T}^d = (\R/\Z)^d$
with a {\it focusing} interaction potential, 
formally given by\footnote{Hereafter,
we use 
$Z$ to denote various normalization constants whose values may change line by line.}
\begin{align}
d\rho(u) = Z^{-1} \exp 
\bigg(\frac{1}{p} \int_{\mathbb{T}^d} |u|^p dx\bigg) d\mu(u), 
\label{Gibbs1}
\end{align}

\noi
where $s>\frac d2$, 
and $p >2$.
Here, $\mu = \mu_{d, s}$ 
denotes  the
massless fractional Gaussian free field
 on $\mathbb{T}^d$, formally given by:
\begin{align}
\begin{split}
d  \mu (u) &  = Z^{-1} e^{-\frac12 \|u \|_{\dot H^s(\T^d)}^2} d u \\
& = Z^{-1} \prod_{n \in \mathbb{Z}^d\backslash\{0\}}  e^{-\frac12  (2\pi |n|)^{2s} |\ft u(n)|^2} d \ft u (n),
\end{split}
\label{gauss1}
\end{align}

\noi
restricted to mean-zero functions on $\T^d$.
We note that,  
when $s=1$, the measure $\mu_{d, 1}$ corresponds to the massless 
Gaussian free field on $\mathbb{T}^d$
(restricted to mean-zero functions on $\T^d$).
See~\cite{LSSW}
for a survey on fractional Gaussian free fields;
see also \cite{She}.
The Gaussian measure $\mu_{d, s}$ 
indeed corresponds to  the induced probability measure under the map:
\begin{align}
\label{series1}
\omega\in \Omega \longmapsto u^\omega(x) = 
\sum_{n \in \mathbb{Z}^d\backslash\{0\}} \frac{g_n (\omega)}{ (2\pi |n|)^s}  e^{2\pi i n\cdot x},
\end{align}

\noi
where 
$\{ g_n \}_{n \in \Z^d\backslash \{0\}}$
is a sequence of  independent standard complex-valued
Gaussian random variables on 
a probability space $(\O,\F,\PP)$.\footnote{In particular,
$\Re g_n$ and $\Im g_n$ are real-valued Gaussian random variables
with mean 0 and variance $\frac 12$.}
Using the random Fourier series representation \eqref{series1}, 
it is easy to see that 
a typical  element $u$  in the support of $\mu_{d, s}$ 
belongs to~$\dot W^{\s, r}(\mathbb{T}^d) \setminus \dot W^{s - \frac d2, r}(\mathbb{T}^d)$ for any 
$\s < s - \frac d2$ and  $1\leq r \leq \infty$, 
where 
  $\dot W^{\s, r}(\mathbb{T}^d)$ denotes the homogeneous Sobolev space, 
restricted to mean-zero functions, 
defined by the norm:
\begin{align*}
\| f \|_{\dot W^{s, r}(\mathbb{T}^d)} = \| D^{s} f \|_{L^r(\mathbb{T}^d)}.
\end{align*}

\noi
Here,  $D^s$ denotes the Riesz potential of order $-s$; see~\eqref{Riesz1}.
In this paper, we restrict our attention 
to the case $s > \frac d2$ such that 
a typical element under $\mu_{d, s}$ is  a function on~$\T^d$
such that no renormalization is required.

Let us first discuss the case 
$s = 1$, which is the most fundamental case
from both the physical and mathematical viewpoints.
In this case, the Gibbs measure $\rho$ in \eqref{Gibbs1}
formally
 corresponds
to the Gibbs measure
of the form:\footnote{Here, we ignore the issue at the zeroth frequency.
See Remark \ref{REM:mean}.}
\begin{align}
d\rho = Z^{-1} e^{- H_\text{NLS}(u)} d u 
\label{Gibbs1a}
\end{align}

\noi
for the 
nonlinear Schr\"odinger equation (NLS) on $\T^d$:
\begin{align}
i \dt u +  \Dl u + |u|^{p-2} u = 0, 
\label{NLS1}
\end{align}

\noi
generated by the Hamiltonian functional $H_\text{NLS}(u)$ given by 
\begin{align*}
H_\text{NLS}(u) = \frac 12\int_{\T^d} |\nb u|^2 d x - \frac 1 p\int_{\T^d} |u|^p d x.
\end{align*}

\noi
The NLS equation \eqref{NLS1}
 has been studied widely as models for describing  
 various  physical phenomena ranging from Langmuir waves in plasmas to signal propagation in optical fibers \cite{SS, KMM, Agrawal}. 
Starting with  seminal works
\cite{LRS} by Lebowitz, Rose, and Speer
and 
\cite{BO94, BO96} by Bourgain, 
the study of the equation \eqref{NLS1}
 from the viewpoint of the (non-)equilibrium statistical mechanics 
 has received extensive  attention;
 see, for example, 
 \cite{BO97,Tzv1,  Tzv2, OQV, LMW, BBulut, 
CFL, DNY2, FOW}.
See also \cite{BOP4} for a survey on the subject, more from the dynamical point of view.

\noi

In the seminal work \cite{LRS}, 
Lebowitz, Rose, and Speer initiated the study 
on the construction of the focusing Gibbs measure
$\rho$ in \eqref{Gibbs1} (and \eqref{Gibbs1a}), 
considering the case $d = s = 1$.
In this case, 
the Gaussian measure $\mu = \mu_{1, 1}$ in \eqref{gauss1} corresponds to the massless
Gaussian free field on $\T$, 
restricted to mean-zero functions, 
and the Gaussian random Fourier series in \eqref{series1}
corresponds to the mean-zero Brownian loop on~$\T$.
In the focusing case, the interaction potential 
$\frac{1}{p} \int_{\mathbb{T}^d} |u|^p dx$ in \eqref{Gibbs1}
is unbounded from above.
In particular,  due to its super-Gaussian growth, 
the density in \eqref{Gibbs1}
is never integrable with respect to the Gaussian measure $\mu$.
Namely, as it is written, the Gibbs measure $\rho$ in \eqref{Gibbs1}
is not normalizable to be a probability measure.
In \cite{LRS}, 
Lebowitz, Rose, and Speer proposed
to instead consider the focusing Gibbs measure 
$\rho = \rho_{d, s, p, K}$
with an $L^2$-cutoff:
\begin{align}
d\rho (u) = Z^{-1}
\ind_{\{\|u\|_{L^2(\mathbb{T}^d)}\le K\}}
\exp \bigg(\frac{1}{p} \int_{\mathbb{T}^d} |u|^p dx \bigg) d\mu (u)
\label{Gibbs2}
\end{align}

\noi
and showed that 
the partition function $Z = Z_{p, K}$
defined by 
\[
Z = Z_{p. K} = \E_\mu \bigg[\ind_{\{\|u\|_{L^2(\mathbb{T}^d)}\le K\}}
\exp \bigg(\frac{1}{p} \int_{\mathbb{T}^d} |u|^p dx \bigg) \bigg]\]

\noi
satisfies
the following dichotomy
for $d = 1$:\footnote{As pointed out by 
Carlen, Fr\"ohlich, and Lebowitz 
\cite[p.\,315]{CFL}, 
there is a gap 
in the proof of \cite[Theorem~2.2]{LRS}.
More precisely, 
 the proof in \cite{LRS} seems to apply only to the case, where the expectation in the definition of $Z_{p,K}$ is taken with respect to a standard (``free'') Brownian motion started at 0, rather than the 
Brownian loop~\eqref{series1} with $s = 1$.} 

\smallskip 

\begin{itemize}
\item[(i)] 
(subcritical case $2 < p < 6$).
We have 
 $Z_{p, K} < \infty $  for any $L^2$-cutoff size $K > 0$,

\smallskip

\item[(ii)] 
(supercritical case $p > 6$).
We have  $Z_{p, K} = \infty$  for any $L^2$-cutoff size $K > 0$.

\end{itemize}

\smallskip

\noi
Furthermore, in the 
critical case ($p = 6$), they showed that  
\[
Z_{6, K}<\infty \ \text{ for } K<\|Q\|_{L^2(\R)}
\qquad
\text{and}
\qquad
Z_{6, K}=\infty \ \text{ for } K>\|Q\|_{L^2(\R)}, 
\]

\noi
where $Q$ is 
 the \textup{(}unique\footnote{Up to the symmetries.}\textup{)} optimizer for the Gagliardo-Nirenberg-Sobolev (GNS) inequality 
on $\R$:
\begin{equation*}
\|u\|_{L^6(\R)}^6 \le C_{\GNS}\|\dx u\|^{2}_{L^2(\R)}
\|u\|^{4}_{L^2(\R)}.
\end{equation*}

\noi
See 
\eqref{GNS1} for the general fractional GNS inequality on $\R^d$, 
which plays a crucial role in the critical case in this paper.
See \cite{BO94, OST1, LW22}
for alternative proofs of (some of) these results.
In a recent work~\cite{OST1}, 
Sosoe, Tolomeo, and the fourth author
revisited this study 
in the critical case ($p = 6$)
and 
proved 
$Z_{6, K} < \infty$
even at the critical $L^2$-threshold $K = \|Q\|_{L^2(\R)}$, 
thus answering an open question posed by Lebowitz, Rose, and Speer
\cite{LRS}
and completing the picture
when $d = s = 1$.
See also a recent work \cite{HN},
where the authors
established
``intermediate integrability''
(stronger than $L^1(d\mu)$ but weaker than 
$ L^p(d\mu)$, $p > 1$)
for the case $K = \|Q\|_{L^2(\R)}$.
In particular, from \cite{LRS, BO94, OST1}, we have
 the following phase transition for the partition
function~$Z_{6, K}$ in terms of the $L^2$-cutoff size $K$:
\begin{align}
Z_{6, K}<\infty \ \text{ for } K\le \|Q\|_{L^2(\R)}
\qquad
\text{and}
\qquad
Z_{6, K}=\infty \ \text{ for } K>\|Q\|_{L^2(\R)}.
\label{phase1}
\end{align}

\noi
We also mention the work \cite{OQ}
which answered another open question by 
Lebowitz, Rose, and Speer
\cite{LRS}
on the construction of the focusing Gibbs measure
restricted to a prescribed value of the $L^2$-norm
(and the momentum).

Our main goal in this paper is to revisit the divergence results
in \cite{LRS, OST1}.
More precisely, 
by following the approach introduced in the recent works \cite{OST, GOTT}
by the first and fourth authors with their collaborators, 
we 
establish  precise  divergence rates
with optimal constants,  as we remove regularization,
in both (i)~the critical case ($p= 6$)
with 
$K>\|Q\|_{L^2(\R)}$
 and (ii)~the supercritical case ($p > 6$).
See Subsection~\ref{SUBSEC:1.2}
for a further discussion; see also Remark \ref{REM:1}.

\begin{remark}\rm
(i) The condition $s > \frac d2$
guarantees that a typical element under the Gaussian measure $\mu_{d, s}$
in \eqref{gauss1}
is  a function (namely, non-singular).
We note that when $s = 1$, 
the problem is non-singular only for 
$d = 1$.

When 
 $s \leq \frac{d}{2}$, 
a typical element in the support of $\mu_{d,s}$ in \eqref{gauss1} is merely a distribution
on $\T^d$ of negative regularity, 
and thus a 
renormalization on the interaction potential $\frac{1}{p} \int_{\mathbb{T}^d} |u|^p dx$
in \eqref{Gibbs2}
is required for constructing the focusing Gibbs measure $\rho$;
see
 \cite{BS96, OOT, OST, OOT21, GOTT}
 and the references therein for the defocusing case.
In particular, in a series
of recent works \cite{OST1, OOT, OST, OOT21, GOTT}, 
Tolomeo and the fourth author with their coauthors
completed the research program, 
initiated by Lebowitz, Rose, and Speer \cite{LRS}
and Bourgain \cite{BO94}, 
on the (non-)construction of the focusing Gibbs measures
on $\T^d$
(with $s = 1$, where the base Gaussian measure
$\mu$ is given by the Gaussian free field on $\T^d$)
for any dimension $d$ and any power $p$.
See also \cite{Nagoji, Nikov}.
We also mention recent works
\cite{RSTW22, 
DRTW, GOW}
on the (non-)construction of focusing Gibbs measures on $\R^d$
with trapping potentials.

\smallskip

\noi
(ii)
In 
\cite{CK12}, the authors
studied the focusing Gibbs measure of type \eqref{Gibbs2}
with $s = 1$ and $p = 4$ on $\T^d$ for $d \ge 3$
via the 
 physical space lattice approximation,
 and established a precise divergence rate
 with a sharp constant
 as the mesh size tends to $0$.
We point out that the divergence observed in~\cite{CK12}
is of different nature than those observed in \cite{OST1, OOT, OST, OOT21}
in the  sense described below.
In~\cite{CK12}, the authors considered the regime $s = 1 \le \frac d2$
such that 
a  continuum limit (even at the level of the base Gaussian measure)
lives on distributions and thus a renormalization
is required. 
However, there was {\it no} renormalization employed
in~\cite{CK12} and thus (rather trivial) divergence emerged 
(whose rate was studied in \cite{CK12}).
Furthermore, in~\cite{OOT21}, 
Okamoto, Tolomeo, and the fourth author proved that 
the $\Phi^3_3$-measure (i.e.~$d = 3$, $s = 1$, and $p = 3$, restricted
to the real-valued setting) is critical, establishing a phase transition.
As a result,  
we expect that the discrete focusing $\Phi^4_3$-model  considered
in~\cite{CK12} does {\it not} have any reasonable continuum limit
even if one introduces a renormalization.
In view of the discussion  above, 
it would be of interest to study 
a divergence rate (as the mesh size tends to~$0$) of 
the discrete focusing $\Phi^4_3$-model
endowed with a proper renormalization.

\end{remark}

Next, let us briefly discuss the 
case $s \ne 1$.
In this case, the Gibbs measure $\rho$ in \eqref{Gibbs1}
can be written in the form
 \eqref{Gibbs1a}, 
where  the Hamiltonian functional $H(u)$ is now given by 
\begin{align}
H(u) = \frac 12 \int_{\T^d} |D^s u|^2 dx - \frac 1p \int_{\T^d} |u|^p dx.
\label{Hamil2}
\end{align}

\noi
One of the most 
important examples of Hamiltonian PDEs generated by the energy functional 
$H(u)$ in \eqref{Hamil2} 
is the following fractional NLS on $\T^d$:
\begin{align}
i\partial_t u - D^{2s} u  + |u|^{p-2 }u = 0.
\label{NLS2}
\end{align}

\noi
When $s = 2$, \eqref{NLS2}
corresponds to the biharmonic NLS
studied in  \cite{OTz, OSTz, OTW}, 
whereas it corresponds to the 
nonlinear half-wave equation
studied in \cite{Poco} 
when $s=\frac12$;
see recent works \cite{ST20,ST21,LW23, FT}
on  the general fractional NLS.
See also \cite{FO23, DNY3, CLL} for  recent developments of this line of research.
We also refer readers to 
\cite[Subsection 1.2]{OST}
for other examples of  dynamical problems
associated with the Hamiltonian $H(u)$ in \eqref{Hamil2}.

In \cite{LW22}, 
the third and fifth authors studied 
the construction of the focusing Gibbs measure~$\rho$ in \eqref{Gibbs2}
 with an $L^2$-cutoff.
In particular, by adapting the arguments in \cite{LRS, OST1}, 
they   proved 
the following dichotomy:
\smallskip 

\begin{itemize}
\item[(i)] 
(subcritical case $2 < p < \frac{4s}{d} + 2$).
We have 
 $Z_{p, K} < \infty $  for any $L^2$-cutoff size $K > 0$,

\smallskip

\item[(ii)] 
(supercritical case $p > \frac{4s}{d} + 2$).
We have  $Z_{p, K} = \infty$  for any $L^2$-cutoff size $K > 0$.

\end{itemize}

\smallskip

\noi
Furthermore, in the 
critical case ($p = \frac{4s}{d} + 2$), they showed that  
\[
Z_{\frac{4s}{d} + 2, K}<\infty \ \text{ for } K<\|Q\|_{L^2(\R^d)}
\qquad
\text{and}
\qquad
Z_{\frac{4s}{d} + 2, K}=\infty \ \text{ for } K>\|Q\|_{L^2(\R^d)}, 
\]

\noi
where $Q$ is 
 the 
 optimizer for the fractional Gagliardo-Nirenberg-Sobolev inequality 
on $\R^d$: see~\eqref{GNS1} 
and  Proposition \ref{LEM:GNS}. 
While we expect that it is possible to adapt
the argument in~\cite{OST1} to study 
normalizability
at the critical $L^2$-threshold $K = \|Q\|_{L^2(\R^d)}$, 
it remains open in the general setting
($s > \frac d2$, $s \ne 1$, and 
$p = \frac{4s}{d} + 2$).

\subsection{Main results}
\label{SUBSEC:1.2}

 Given  $N\in \N$, we define the  frequency projector $\P_{N}$  
 onto the frequencies $\{ |n|\le N\}$ by 
\begin{align}
\P_{N} f = \sum_{ |n| \leq N}  \widehat f (n)  e^{2\pi in\cdot x}
\label{P1}
\end{align}
\noi
for a function $f$  on $\mathbb{T}^d$. 
Then, 
we define  the truncated  Gibbs measure $\rho_N = \rho_{d,s, p,K,N}$ by
setting
\begin{align}
\begin{split}
d \rho_N (u) 
 = Z_{K,N}^{-1} 
\ind_{\{\|\P_{N} u\|_{L^2(\mathbb{T}^d)}\le K\}}\exp 
\bigg({\frac 1 p\int_{\mathbb{T}^d}  | \P_{N} u|^p   dx} \bigg) d \mu (u), 
\end{split}
\label{GibbsN}
\end{align}

\noi
where 
$Z_{K, N} = Z_{d, s, p, K, N}$
denotes the   partition function for $\rho_N$,  given by 
\begin{align}
Z_{K, N} = 
\mathbb{E}_{\mu}\bigg[
\ind_{\{ \| \P_{N} u \|_{L^2 (\mathbb{T}^d)}\leq K \}} \exp\bigg( \frac 1 p\int_{\mathbb{T}^d}  | \P_{N} u|^p   dx \bigg)\bigg].
\label{Z1} 
\end{align}

We now state our main results.
The first result is on the $L^2$-critical case.

\begin{theorem}[critical case]\label{THM:1}
Let $d,  s \in \N$ such that   $s>\frac{d}{2}$
and  $p=\frac{4s}{d}+2$.
Let  $Q$ 
 the 
 optimizer for the fractional Gagliardo-Nirenberg-Sobolev inequality 
\eqref{GNS1}
on $\R^d$.
Then, given any   $K> \|Q\|_{L^2 (\R^d)}$, 
we have 
\begin{align}
\begin{split}
\log Z_{K, N} 
& = 
\CC_{K} N^{2s}+o(N^{2s})\\
& = 
\CC_{K} N^{\frac{dp}{2}-d}+o(N^{\frac{dp}{2}-d}),    
\end{split}
\label{A1}
\end{align}

\noi
as $N \to \infty$.
Here, 
 $\CC_K = \CC_{K}(d, s, p, K) > 0$ is 
given by 
\begin{align}
\CC_K= - \inf_{\substack{\|u\|_{L^2(\R^d)} = K\\ 
u = \P u}} H_{\R^d}(u), 
\label{A2}
\end{align}

\noi
where $H_{\R^d}(u)$ is the Hamiltonian functional on $\R^d$\textup{:}
\begin{align}
H_{\mathbb{R}^d} (u ) = \frac{1}{2} \int_{\R^d} |D^s u|^2 dx
- 
\frac{1}{p} \int_{\R^d} |u|^p dx,
\label{Hamil3}
\end{align}

\noi
and $\P$ denotes the ball multiplier given by 
\begin{align}
\F_{\R^d}(\P f)(\xi)=\ind_{\left\{\xi|\le 1\right\}}\widehat{f}(\xi).
\label{ball0}
\end{align}

\end{theorem}

\medskip

We present a proof of Theorem \ref{THM:1} in Section \ref{SEC:3}, 
where we prove the following upper bound:
\begin{align}
\limsup_{N\to \infty}N^{-2s} \log Z_{K, N}\le \CC_K
\label{A2a}
\end{align}

\noi
in Subsection \ref{SUBSEC:3.1}, 
while the lower bound 
\begin{align}
\liminf_{N\to \infty}N^{-2s} \log Z_{K, N}\ge \CC_K
\label{A2b}
\end{align}

\noi
is established 
in Subsection \ref{SUBSEC:3.2}. 

\begin{remark}\label{REM:1}\rm

(i) 
We point out that the assumption  $s \in \N$
is needed only for establishing the upper bound \eqref{A2a}
(see Remark \ref{REM:up1}), 
whereas we prove the lower bound \eqref{A2b}
for general $s > \frac d2$.
While we expect that the upper bound \eqref{up0}
also holds for general $s > \frac d2$, 
we do not pursue this issue
in the present paper
for conciseness of the presentation.
Lastly, we point out that 
even with the restriction $s \in \N$, 
Theorem \ref{THM:1} covers
the most important case $d = s = 1$
(with $p = 6$), 
studied intensively in \cite{LRS, BO94, OST1}, 
providing the following refined description
of  the  phase transition for the partition
function~$Z_{6, K}$ in terms of the $L^2$-cutoff size $K$:
\begin{align*}
\lim_{N \to \infty} N^{-2s}\log Z_{K,N}
= 
\begin{cases}
0 & \text{for } K \leq \|Q\|_{L^2(\R)}, \\
\CC_K & \text{for } K > \|Q\|_{L^2(\R)}, 
\end{cases}
\end{align*}

\noi
as compared to 
 \eqref{phase1}.

\medskip

\noi
(ii)  When    $K \le  \|Q\|_{L^2 (\R^d)}$, 
it is easy to see, 
using the  
fractional Gagliardo-Nirenberg-Sobolev inequality~\eqref{GNS1}, 
 that  
$\CC_K$ in \eqref{A2} is non-positive
(namely, the Hamiltonian is non-negative).
In this case, 
 the formula \eqref{A1}
does not provide any information,
which is indeed consistent with 
the fact that $\sup_{N \in \N} Z_{K. N}  < \infty$, 
at least for    $K <  \|Q\|_{L^2 (\R^d)}$
(also with an equality
for $d = s = 1$).

\end{remark}

Next, we state our second result on  the $L^2$-supercritical case.

\begin{theorem}[supercritical case]\label{THM:2}

Let $d \in \N$, $s > \frac d2$, and 
 $p>\frac{4s}{d}+2$.
 Then, given any $K > 0$, 
we have 
\begin{align*}
\log Z_{K, N}= C_B \frac{K^p}{p} N^{\frac{dp}{2}-d}+o(N^{\frac{dp}{2}-d}), 
\end{align*}

\noi
as $N\to \infty$.
Here,  $C_B$ is the optimal constant for the following Bernstein inequality on $\R^d$\textup{:}
\begin{align}\label{A4}
\|\P f\|^p_{L^{p}(\R^d)}\le C_B\|f\|^p_{L^{2}(\R^d)},
\end{align}

\noi
where  $\P$ is  the ball multiplier defined in  \eqref{ball0}.

\end{theorem}

Theorems \ref{THM:1} and \ref{THM:2}
revisit the non-normalizable regime, 
previously studied in 
 \cite{LRS, OST1, LW22}, 
and provide precise divergence rates
with optimal constants.
In particular, 
when 
 $d = s = 1$ and $p = 6$, 
 Theorem \ref{THM:1}
provides a complete picture of 
the  phase transition of the $L^2$-critical focusing Gibbs measure
at the critical $L^2$-threshold $K = \|Q\|_{L^2(\R)}$, 
previously studied in  \cite{LRS, OST1}.

Our proofs of 
Theorems \ref{THM:1} and \ref{THM:2}
are based on the Bou\'e-Dupuis variational formula
(Lemma~\ref{LEM:var}), recently popularized by 
Barashkov and Gubinelli \cite{BG}
in the construction of the $\Phi^4_3$-measure.
In establishing the lower bounds (see  \eqref{A2b} and \eqref{sp1b}), 
we follow the approach introduced
in a series of recent works
\cite{TW23, OOT, OST, OOT21},   
where analogous
non-normalizability results were established in various settings;
see also 
\cite{LW22,  GOTT}.
The heart of this approach lies in 
constructing a specific drift $\ta = \ta_N$ 
in applying  
the Bou\'e-Dupuis variational formula (Lemma~\ref{LEM:var})
for each $N \in \N$
such that 
$\P_N \Dr (1)$ defined in \eqref{var3} looks like
\begin{align}
 \text{`$-Y_N$+ a deterministic perturbation',}
\label{A5}
\end{align}

\noi
 where $Y_N$ is as in \eqref{Y2} (see also \eqref{var2})
and 
 the perturbation term  is bounded in 
$L^2(\T^d)$ but  has a large $L^p$-norm, 
leading to the desired divergence;
see Lemma \ref{LEM:Z} for the construction of
a smooth process $\ze_N$, approximating~$Y_N$.

In \cite{OST}, 
Seong, Tolomeo, and the fourth author
refined this approach to obtain 
a precise divergence rate 
with a sharp constant
for the logarithmically correlated Gibbs measure 
with a focusing quartic interaction potential.
More precisely, 
they 
used
the periodized version of  
a (suitably dilated and frequency-truncated)\footnote{See \eqref{low1} and \eqref{sp2}.}
almost optimizer of 
Bernstein's inequality \eqref{A4} on $\R^d$
(with $p = 4$)
as the deterministic perturbation in \eqref{A5}, 
which allowed them to 
establish a precise divergence rate 
with  the sharp constant
given by 
 the optimal constant $C_B$ for Bernstein's inequality 
 \eqref{A4} (with $p = 4$).
Our proof of  
Theorem \ref{THM:2}
on 
the  supercritical case 
closely follows the argument in \cite{OST}.
In the critical case 
(Theorem \ref{THM:1}), 
the divergence occurs only 
for large $L^2$-norms (namely, $K > \|Q\|_{L^2(\R^d)}$).
Thus, 
we need to carry out more careful
analysis, 
using the sharp fractional GNS inequality \eqref{GNS1}
and its optimizer $Q$ (see Remark \ref{REM:GNS})
which plays a crucial role in establishing a key lemma
(Lemma \ref{LEM:sol2}).
Lastly, 
we point out that the ball multiplier
$\P$ defined in \eqref{ball0}
appears in both Theorems \ref{THM:1} and \ref{THM:2}
due to our use of the sharp frequency projector
$\P_N$ defined in \eqref{P1}
as regularization.
A different regularization procedure leads
to a different optimal constant but the divergence
rate should remain the same.

\begin{remark}\label{REM:mean}
\rm

Let $\wt \mu = \wt \mu_{d, s}$ be 
the {\it massive}
fractional Gaussian free field on $\T^d$
with 
the formal density: 
\begin{align}
d \wt \mu =  Z^{-1} e^{-\frac 12 \|u\|_{H^s(\T^d)}^2} d u
\label{gauss2}
\end{align}

\noi
which is nothing but   the induced probability measure under the map:
\begin{align}
\omega\in \Omega \longmapsto u^\omega(x) = 
\sum_{n \in \mathbb{Z}^d} \frac{g_n (\omega)}{ \jb{n}^s}  e^{2\pi i n\cdot x},
\label{series2}
\end{align}

\noi
Here,  $\jb{n} = (1 + 4\pi^2 |n|^2)^\frac{1}{2}$
and 
$\{ g_n \}_{n \in \Z^d}$
is a sequence of  independent standard complex-valued
Gaussian random variables.
When $d = s= 1$, 
the series in \eqref{series1}
is  the mean-zero Brownian loop, 
while the series in \eqref{series2}
corresponds to 
 the Ornstein-Uhlenbeck loop.
Then, instead of~\eqref{Gibbs2}, 
we can consider the focusing Gibbs measure 
with 
the massive fractional Gaussian free field
$\wt \mu$ in~\eqref{gauss2} as the base Gaussian measure:
\begin{align}
d\wt \rho (u) = Z^{-1}
\ind_{\{\|u\|_{L^2(\mathbb{T}^d)}\le K\}}
\exp \bigg(\frac{1}{p} \int_{\mathbb{T}^d} |u|^p dx \bigg) d\wt \mu (u)
\label{Gibbs3}
\end{align}

\noi
which, as observed in \cite{BO94},  is a more natural Gibbs measure
for the (fractional) NLS \eqref{NLS1} (and~\eqref{NLS2})
due to the lack of the conservation of the spatial mean under the dynamics.
We note that  the results in \cite{LRS, OST1}
also hold for $\wt \rho$ defined in \eqref{Gibbs3}
with the massive fractional Gaussian free field
$\wt \mu$; see \cite[Remark~1.2]{OST1}.

Let 
$\wt Z_{K, N}$ denote the partition function
for the truncated Gibbs measure associated 
with the massive fractional Gaussian free field $\wt \mu$, 
 given by
\begin{align*}
\wt Z_{K, N} = 
\E_{\wt \mu}\bigg[
\ind_{\{ \| \P_{N} u \|_{L^2 (\mathbb{T}^d)}\leq K \}} \exp\bigg( \frac 1 p\int_{\mathbb{T}^d}  | \P_{N} u|^p   dx \bigg)\bigg].
\end{align*}

\noi
Namely, 
$\wt Z_{K, N}$
is given by 
 \eqref{Z1}, 
where we replace $\mu$  by $\wt \mu$.
Then, it is easy to check that 
an analogue of Theorem~\ref{THM:2} 
in the supercritical case  holds
for $\wt Z_{K, N}$:
\begin{align*}
\log \wt Z_{K, N}= C_B \frac{K^p}{p} N^{\frac{dp}{2}-d}+o(N^{\frac{dp}{2}-d}), 
\end{align*}

\noi
as $N \to \infty$, where $C_B$ is as in \eqref{A4}.
Moreover, by examining its proof, 
we also see that 
an analogue of  Theorem \ref{THM:1}
in the critical case
also holds
for $\wt Z_{K, N}$:
\begin{align}
\log \wt Z_{K, N} 
 = 
\CC_{K} N^{2s}+o(N^{2s})
 = 
\CC_{K} N^{\frac{dp}{2}-d}+o(N^{\frac{dp}{2}-d}),    
\label{ZA2}
\end{align}

\noi
as $N \to \infty$, 
where $\CC_K$ is as in \eqref{A2}.
See Remark \ref{REM:non1} for a further discussion.

We also point out that 
Theorems \ref{THM:1} and \ref{THM:2} also hold in the real-valued setting, 
which  is, for example,  relevant to the study of the generalized KdV equation
on $\T$:
\begin{align*}
\dt u + \dx^3 u + \dx (u^{p-1}) = 0.
\end{align*}

\noi
See 
\cite{ORTh, CK} for a further discussion.

\end{remark}

\section{Notations and preliminary lemmas}

\subsection{Notations}
Let $A\les B$ denote an estimate of the form $A\leq CB$ for some constant $C>0$. We write $A\sim B$ if $A\les B$ and $B\les A$, while $A\ll B$ denotes $A\leq c B$ for some small constant $c> 0$. 
We use $C>0$ to denote various constants, which may vary line by line.

Given $r > 0$, we use $B_r \subset \R^d $ to denote the closed ball of radius $r$ centered at the origin:
\begin{align}
B_r = \{ \xi \in \R^d: |\xi|\le r\}.
\label{ball1}
\end{align}

We denote by $\Law(X)$ the law of
a random variable $X$. 
In the following, we fix $d$, $s$, and $p$,  and thus
we often drop dependence on these parameters.
For example, 
we use the following short-hand notation:
\begin{align*}
\mu = \mu_{d, s} \qquad\text{and}\qquad
Z_{K, N} = Z_{d, s, p, K, N}, 
\end{align*}

\noi
where  $\mu_{d, s}$ and $Z_{d, s, p, K, N}$ are
as  in \eqref{gauss1} and \eqref{Z1}.

We use $\ft f$ to denote the Fourier transform of a function $f$ on $\T^d$, 
while we use $\F_{\R^d}(f)$ 
to denote the Fourier transform of a function $f$ on $\R^d$.
Given $s \in \R$, we define the Riesz potential of order $-s$ 
by 
\begin{align}
D^s f (x) = \sum_{n \in \mathbb{Z}^d \backslash \{ 0 \}} 
(2\pi | n|)^s \ft f (n) e^{2\pi i n \cdot x} 
\label{Riesz1}
\end{align}

\noi
and define the homogeneous Sobolev space $\dot H^s(\T^d)$ via the norm:
\begin{align*}
\| f \|_{\dot H^{s}(\mathbb{T}^d)} = \| D^{s} f \|_{L^2(\mathbb{T}^d)}
=\bigg(\sum_{n\in \Z^{d}}(2\pi |n|)^{2s}|\widehat{f}(n)|^2\bigg)^\frac 12 .
\end{align*}

\noi
Similarly, given $ s > 0$, we define $\dot H^s(\R^d)$ by the (semi-)norm:
\begin{align*}
\|f\|_{\dot H^{s}(\R^d)}
 = \| D^{s} f \|_{L^2(\R^d)}
=\|\mathcal{F}_{\R^d}^{-1}((2\pi |\xi|)^{s}\widehat{f}(\xi))\|_{L^{2}(\R^d)}, 
\end{align*}

\noi
where we impose $\lim_{|x|\to \infty} f(x) = 0$.
We use $\S(\R^d)$ to denote the Schwartz class on $\R^d$.

Let $s \in \N$. Then, by Plancherel's identity and 
 the multinomial theorem, we have 
\begin{align}
\begin{split}
\|f\|^2_{\dot{H}^{s}(\R^d)}&=
\int_{\R^d} 
(2\pi |\xi|)^{2s}|\F_{\R^d}(f)(\xi)|^2 d\xi
=\sum_{|\alpha|=s}\binom{s}{\alpha}
\int_{\R^d} 
(2\pi \xi)^{2\al}|\F_{\R^d}(f)(\xi)|^2 d\xi\\
&=\sum_{|\alpha|=s}\binom{s}{\alpha}\|\dd^\al f\|^2_{L^{2}(\R^d)}.
\end{split}
\label{Hs4}
\end{align}

\noi
We note that 
$\dd^\al$ is a local operator, 
which plays a crucial role 
in Subsection \ref{SUBSEC:3.1}.

\subsection{Variational formulation}

In this subsection, 
we recall 
the  Bou{\'e}-Dupuis
variational formula 
(Lemma \ref{LEM:var})
which was recently popularized
by Barashkov and Gubinelli \cite{BG}
in the construction of the $\Phi^4_3$-measure.

 Let $W (t)$ denote a mean-zero cylindrical 
 Wiener process in $L^2 (\mathbb{T}^d)$:
\begin{align}
W (t) = \sum_{n \in \mathbb{Z}^d \backslash \{ 0 \}} B_n (t)
e^{2\pi i n \cdot x},
\label{W1}
\end{align}

\noi
where $\{ B_n \}_{n \in \mathbb{Z}^d \backslash \{ 0 \}}$ is a sequence of
mutually independent complex-valued Brownian motions.\footnote{By convention, 
we assume that $\text{Var}(B_n(t)) = t$.} 
Then,  define a 
Gaussian process $Y (t)$ by setting
\begin{equation}
\label{Y1} 
Y (t) = D^{- s} W (t) = \sum_{n \in \mathbb{Z}^d\backslash \{ 0 \}} \frac{B_n (t)}{(2\pi | n |)^s} e^{2\pi i n \cdot x} .
\end{equation}

%
%

\noi
In particular, we have $\Law (Y (1)) = \mu$,
where $\mu = \mu_{d, s}$ is the massless Gaussian 
free field defined in 
\eqref{gauss1} (see also \eqref{series1}).
In the following, we restrict our attention to the case $s > \frac d2$
such that $Y(1)$ is almost surely a function on $\T^d$.

Given  $N \in \N$, we set 
\begin{equation}
Y_N = \P_{N} Y (1), 
\label{Y2}
\end{equation}

\noi
where $\P_N$ is as in \eqref{P1}.
Then, we have 
\begin{align}
\label{Y3}
\E\big[\|Y_N\|^2_{\dot H^{s}(\T^d)}\big]
\sim \sum_{0<|n|\le N}1
\sim N^d.
\end{align}

\quad 
Let $\mathbb{H}_a$ be the space of drifts, consisting of
spatially mean-zero progressively 
measurable\footnote{with respect to  the filtration generated by $Y$.}
processes
in $L^2 ([0, 1] ; L^2 (\mathbb{T}^d))$, $\PP$-almost surely.
Given a drift $\dr \in \Ha$, 
define
 $\Theta$ by 
\begin{equation}
\Theta = \int_0^1 D^{- s} \theta (t) dt
\label{var3}
\end{equation}

\noi
and $\Dr_N = \P_N \Dr(1) $ as in \eqref{Y2}.
We now recall the Bou\'e-Dupuis variational formula
\cite{BD98,U14, Zhang09};
see also \cite[Appendix A]{TW23}.

\begin{lemma}
\label{LEM:var} 
 Suppose that $F: C^\infty(\T^d) \to \mathbb{R}$ is
a bounded measurable function.
Then, we have
\begin{align}\label{var2} 
\log\E \big[e^{F (Y_N)}\big] = \sup_{\theta \in\Ha} 
\E \bigg[ F (Y_N + \Theta_N) - \frac{1}{2} \int_0^1\| \theta (t) \|^2_{L^2 (\mathbb{T}^d)} d t \bigg], 
\end{align}

\noi
where 
 the expectation $\mathbb{E}=\mathbb{E}_{\PP}$ is taken with respect to
the underlying probabilistic measure $\PP$.
\end{lemma}

In our application, we set $F$ by 
\[F (u) = 
 \frac 1 p \| \P_N u\|^p_{L^p(\T^d)} \cdot 
 \ind_{\{ \| \P_{N} u \|_{L^2 (\mathbb{T}^d)}\leq K \}}\]

\noi
for given $K > 0$ and $N \in \N$, 
which is bounded in view of Bernstein's inequality.
We recall the following bound:
\begin{equation}
\label{var4} 
\| \Theta \|^2_{\dot{H}^s (\mathbb{T}^d)}
\leq \int_0^1 \| \theta (t) \|^2_{L^2 (\mathbb{T}^d)} d t
\end{equation}

\noi
for any $\ta \in \Ha$, 
which follows from 
Minkowski's integral inequality followed by Cauchy-Schwarz's inequality;
see \cite[Lemma 4.7]{GOTW}.

In establishing a lower bound 
for  the partition function $Z_{K, N}$
(such as \eqref{A2b}), 
we construct a specific drift $\ta = \ta_N$ 
in applying  Lemma \ref{LEM:var}.
As in \cite{OST, LW22, GOTT}, 
we choose a drift $\ta = \ta_N$  such that 
$\Dr_N$ looks like
`$-Y_N$+ a deterministic perturbation', where the perturbation term is bounded in 
$L^2(\T^d)$ but  has a large $L^p$-norm.
The next lemma, originally introduced
in \cite[Lemma 2.1]{TW23}
and  \cite[Lemma 3.4]{OST}, 
 provides
the construction of a good approximation of $Y_N$;
see \cite[Lemma 3.5]{LW22}
for the proof.

\begin{lemma}
\label{LEM:Z} 
Given $s > \frac{d}{2}$ and  $N \in \N$, 
define $\ze_N (t)$ by its Fourier coefficients. 
For  $0< | n | \leq N$,  
let 
$\ft \ze_N (t, n)$ be the solution to the following differential equation{\rm{:}}
\[ 
\begin{cases}
d \ft \ze_N (t, n) = (2\pi | n |)^{- s} N^{\frac{d}{2}}
(\ft Y (t, n) - \ft \ze_N (t, n)) d t\\
\ft \ze_N (0, n)  = 0,
\end{cases} 
\]

\noi
and we set $\ft \ze_N (t, n) = 0$ for $n = 0$ or $| n | > N$. 
Here, $Y$ is as in \eqref{Y1}.
Then, given any $\eps > 0$, 
 we have 
\begin{align*}
 \mathbb{E} \Big[\| \ze_N (1) - Y_N \|^p_{L^p (\mathbb{T}^d)}\Big] 
 & \lesssim 
\big\{\max
(N^{- s + \frac{d}{2}}, N^{- \frac{d}{2} +\eps})\big\}^{\frac{p}{2}}, \\
 \mathbb{E} \bigg[ \int_0^1 \Big\|  \frac{d}{dt} \ze_N (t) \Big\|_{\dot H^s (\mathbb{T}^d)}^2 dt \bigg] 
& \lesssim \max
(N^{\frac{3 }{2}d - s}, N^{\frac{d}{2} +\eps }).
\end{align*}

\noi
for any $N \gg 1$ and  finite $ p \geq 1$.

\end{lemma}

\subsection{Fractional Gagliardo-Nirenberg-Sobolev inequality}

In this subsection, we establish some preliminary results
needed to study the critical case (Theorem \ref{THM:1})
whose proof is presented in Section \ref{SEC:3}.

We first recall the 
following fractional GNS inequality on $\mathbb{R}^d$:
\begin{align}
\|u\|^p_{L^{p}(\R^d)}\le C_{\GNS}(d,p,s)	\| u \|^{\frac{(p - 2) d}{2s}}_{\dot{H}^s (\mathbb{R}^d)} \| u \|_{L^2 (\mathbb{R}^d)}^{2 + \frac{p-2}{2 s} (2 s - d)},
\label{GNS1}
\end{align}

\noi
where
$C_{\GNS} = C_{\GNS}(d,p,s)$ denotes the optimal constant;
see \eqref{GNS3}.
As in \cite{LRS, OST1, LW22}, optimizers for \eqref{GNS1}
play a crucial role in our argument.
See
\cite{Nagy} 
and 
\cite{Weinstein}
for the proof of the following lemma
for (i) $d = s = 1$ and 
(ii)  $d\ge 2$ and $s=1$, respectively.
For the general case, see, for example, 
 \cite[Proposition 3.1]{Frank} and \cite[Theorem 2.1]{BFV14} and the references therein.

\begin{lemma}
\label{LEM:GNS}
Let $d \geq 1$.
Given $p > 2$ and $s > 0$, satisfying
$ p < \frac{2 d}{d - 2 s}$ if $d >2 s$, 
consider the functional
\begin{equation}
\label{GNS2} 
J_{\R^d} (u) = 
\frac{\| u \|^{\frac{(p - 2) d}{2s}}_{\dot{H}^s (\mathbb{R}^d)} \| u \|_{L^2 (\mathbb{R}^d)}^{2 + \frac{p -2}{2 s} (2 s - d)}}{\| u \|^{p}_{L^p (\mathbb{R}^d)}}
\end{equation}

\noi
on $H^s (\mathbb{R}^d)$. 
Then, the minimum
\begin{equation}
\label{GNS3} 
C_{\GNS}^{- 1} =
C_{\GNS} (d, p, s)^{- 1} :=
\inf_{\substack{ u \in H^s (\mathbb{R}^d)\\u \neq 0}} J_{\R^d} (u)
\end{equation}

\noi
is attained at some function $Q \in H^s (\mathbb{R}^d)$.
\end{lemma}

\begin{remark}
\label{REM:GNS}
\rm

As pointed out 
in  \cite[Remark 2.2]{LW22},
given any $c \in \mathbb{R}\backslash \{ 0 \}$, $b > 0$, and $a \in \mathbb{R}^d$, 
the  function 
$ c Q (b (x - a))$
is also a minimizer of the functional $J_{\R^d}(u)$ defined in \eqref{GNS2}. 
Therefore, by choosing the parameters appropriately, we assume
that
\begin{align}
\label{GNS4} 
\| Q \|_{L^2 (\mathbb{R}^d)}  = \| Q \|_{\dot{H}^s (\mathbb{R}^d)}
\quad \text{ and }\quad\| Q \|_{\dot{H}^s (\mathbb{R}^d)}^2 = \frac{2}{p} \| Q \|_{L^p(\mathbb{R}^d)}^p 
\end{align}

\noi
in the following.
Then, 
we have  $H_{\mathbb{R}^d} (Q) = 0$, 
where 
$H_{\mathbb{R}^d} (u)$ is as in \eqref{Hamil3}.
Moreover,  from~\eqref{GNS2} and \eqref{GNS4}, we have
\begin{equation}
\label{GNS5} 
C_{\GNS} = \frac{p}{2} \| Q
\|^{2 - p}_{L^2 (\mathbb{R}^d)}.
\end{equation}
\end{remark}

\noi
In the following, we assume that \eqref{GNS4} and \eqref{GNS5} are satisfied.

\medskip

Let $f$ be  a nice function  on $\R^d$.
Then, given $\eps > 0$, 
define $f_\eps$ by 
\begin{align}
f_\eps(x) = \eps^{-\frac d2}f(\eps^{-1} x)
\label{per1}
\end{align}

\noi
for $x \in \R^d$, 
and define the periodization $f^\per$ of $f$ by 
\begin{align}
f^\per(x) = \sum_{n \in \Z^d} f(x + n)
\label{per2}
\end{align}

\noi
for $x \in \T^d$.
Given small $\eps > 0$, 
we  then define a function  $f_\eps^\per$ on $\T^d$ by setting
\begin{align}
f_\eps^\per = (f_\eps)^\per.
\label{per3}
\end{align}

\noi
Then, 
from \eqref{per3} with \eqref{per1} and \eqref{per2}, 
we have 
\begin{align}
\ft f_\eps^\per(n) = \eps^\frac d2 \F_{\R^d} (f) (\eps n)
\label{per4}
\end{align}

\noi
for any $n \in \Z^d$.
Then, we have the following lemma.

\begin{lemma}
\label{LEM:sol1} 
Let $s > 0 $ and $2 \leq p < \infty$
such that $\frac sd \ge \frac 12 - \frac 1p$.
Given  a Schwartz function $f\in  \S(\R^d)$, 
let $f_\eps^\per$ be as in \eqref{per3}.
Then, we have 
\begin{align}
\| f_\eps^\per  \|^2_{\dot{H}^s (\mathbb{T}^d)}
&=  \eps^{-2s} \| f \|^2_{\dot H^{s} (\R^d)} + o(\eps^{-2s}), 
\label{SO1}\\
\| f_\eps^\per  \|_{L^p (\mathbb{T}^d)}^p &=
\eps^{-\frac{d p}{2} + d} \| f \|^p_{L^p (\R^d)} + o (\eps^{-\frac{d p}{2}+d}), 
  \label{SO2}
\end{align}

\noi
as $\eps \to \infty$, 
where the convergence rates of the error terms depend on $f$.
	
\end{lemma}

\begin{proof}
From \eqref{per4}
and a Riemann sum approximation, we have 
\begin{align*}
\| f_\eps^\per\|_{\dot H^s(\T^d)}^2
= 
\eps^{-2s} 
\sum_{n \in \Z^d} (2\pi \eps |n|)^{2s} 
 |\F_{\R^d} (f) (\eps n)|^2 \cdot \eps^d
 = \eps^{-2s} 
 \| f \|^2_{\dot H^{s} (\R^d)} + o(\eps^{-2s}), 
\end{align*}

\noi
as $\eps \to 0$.
This proves \eqref{SO1}.

Let $R=  [- \frac 12, \frac12)^d\subset \R^d$.
Then, from \eqref{per2}, we have 
\begin{align}
\| f_\eps^\per  \|_{L^p (\mathbb{T}^d)}^p 
=
\| f_\eps  \|_{L^p (R)}^p 
+ O\Big(\| f_\eps  \|_{L^p (R)}^{p-1}
\| f_\eps  \|_{L^p (R^c)} 
+ \| f_\eps  \|_{L^p (R^c)}^p \Big).
\label{SO3}
\end{align}

\noi
By the Lebesgue dominated convergence theorem, 
we have
\begin{align}
\begin{split}
\eps^{\frac{d p}{2} - d} \| f_\eps  \|_{L^p (R)}^p 
& = 
 \| f \|^p_{L^p (\eps^{-1}R)}
=  \| f \|^p_{L^p (\R^d)} + o(1), \\
\eps^{\frac{d p}{2} - d} \| f_\eps  \|_{L^p (R^c)}^p 
& = 
 \| f \|^p_{L^p ((\eps^{-1}R)^c)}
=   o(1), 
\end{split}
\label{SO4}
\end{align}

\noi
Then, \eqref{SO2} follows from \eqref{SO3} and \eqref{SO4}.
\end{proof}

Given  $K> \|Q\|_{L^2 (\R^d)}$, define the energy functional $\EE_K$ by 
\begin{align}
\EE_K(u)
= \frac{K^p}{p}\|u\|^{p}_{L^p(\T^d)}-\frac{K^2}{2}\|u\|^2_{\dot H^{s}(\T^d)}.
\label{QQ1}
\end{align}

\noi
Then, we have the following positivity result for $\EE_K$.

\begin{lemma}
\label{LEM:sol2}
Let $d \ge 1$, $s > 0$, and $p > 2$
such that 
 $p=\frac{4s}{d}+2$.
Let $Q$ be the optimizer of the GNS inequality \eqref{GNS1}
on $\R^d$, 
defined in Remark \ref{REM:GNS}.
Then, given any   $K> \|Q\|_{L^2 (\R^d)}$, 
we have
\begin{align}
\sup_{\substack{\|u\|_{L^{2}(\T^d)}=1\\u=\P_Nu}}
\EE_K(u) > 0, 
\label{QQ2}
\end{align}

\noi
provided $N \in \N$ is sufficiently large.
\end{lemma}

\begin{proof}
Fix $K > \|Q\|_{L^2(\R^d)}$.
In view of 
\eqref{GNS5}, choose small $\g > 0$
such that 
\begin{align}
\frac p 2 K^{2 - p} (C_{\GNS}^{-1} + 3\g) < 1.
\label{QQ3}
\end{align}

Let $J_{\R^d}(u)$ be as in  \eqref{GNS2}.
Under the current assumption, we have
\begin{equation}
J_{\R^d} (u) = 
\frac{\| u \|^{2}_{\dot{H}^s (\mathbb{R}^d)} \| u \|_{L^2 (\mathbb{R}^d)}^{p - 2}}{\| u \|^{p}_{L^p (\mathbb{R}^d)}}.
\label{QQ4}
\end{equation}

\noi
We define $J_{\T^d}(u)$
by replacing $\R^d$ with $\T^d$ in \eqref{QQ4}.
Let  $Q$ be 
the optimizer of the GNS inequality~\eqref{GNS1} on $\R^d$
as in Remark \ref{REM:GNS}
such that 
$J_{\R^d} (Q) = C_\GNS^{-1}$.
Given $\g > 0$ as in \eqref{QQ3}, 
let $f \in \S(\R^d)$ such that 
\begin{align*}
J_{\R^d}(f) \le   C_\GNS^{-1} + \g.
\end{align*}


\noi
Then, from Lemma \ref{LEM:sol1}, 
we have 
\begin{align}
J_{\T^d}(f_{\eps}^\per)\le  C^{-1}_{\GNS}+2\gamma
\label{QQ5}
\end{align}

\noi
by choosing $\eps = \eps(\g)> 0$ sufficiently small, 
where  $f_{\eps}^\per$ is as in \eqref{per3}.

Given $N \in \N$, 
define a function $u_N$ on $\T^d$ by setting
\begin{align}
u_N=\frac{\P_{N}f_{\eps}^\per}{\|\P_{N}f_{\eps}^\per\|_{L^{2}(\T^d)}}.
\label{QQ6}
\end{align}

\noi
Recalling from Lemma \ref{LEM:GNS} and \eqref{per3}
that  $f_\eps^\per \in H^s(\T^d)$, 
we see that $\P_{N}f_{\eps}^\per$ converges to $f_{\eps}^\per$
in $H^s(\T^d)$ as 
$N \to \infty$.
Then, from Sobolev's inequality, 
we see that 
$\P_{N}f_{\eps}^\per$ converges to $f_{\eps}^\per$
also in $L^p(\T^d)$ as 
$N \to \infty$, 
where
$p=\frac{4s}{d}+2$.
Hence, it follows from \eqref{QQ5}
that 
\begin{align}
J_{\T^d}(\P_{N}f_{\eps}^\per)
\le 
J_{\T^d}(f_{\eps}^\per) + 2\g \le  C^{-1}_{\GNS}+3\g, 
\label{QQ7}
\end{align}

\noi
provided $N = N(\g, \eps) = N(\g)$
is sufficiently large.
Hence, 
from 
\eqref{QQ1}, \eqref{QQ6}, \eqref{QQ4} (for $J_{\T^d}$), \eqref{QQ7}, 
and \eqref{QQ3}, we obtain
\begin{align*}   
\EE_K(u_N) 
&= \frac{K^{p}}{p}\|u_N\|^{p}_{L^{p}(\T^d)}\bigg(1-\frac{p}{2}K^{2-p}\frac{\|u_N\|^2_{\dot H^{s}(\T^d)}}{\|u_N\|^p_{L^{p}(\T^d)}}\bigg)\notag\\
&=\frac{K^{p}}{p}\|u_N\|^{p}_{L^{p}(\T^d)}
\bigg(1-\frac{p}{2}K^{2-p}J_{\T^d}(\P_{N}f_{\eps}^\per)\bigg)\notag\\
& \ge \frac{K^{p}}{p}\|u_N\|^{p}_{L^{p}(\T^d)}
 \bigg(1-\frac{p}{2}K^{2-p}( C^{-1}_{\GNS}+3\gamma)\bigg) > 0.
\end{align*}

\noi 
This proves \eqref{QQ2}.
\end{proof}

\section{Critical case} 
\label{SEC:3}

In this section, we present a proof of Theorem \ref{THM:1}
on the $L^2$-critical case:
\begin{align}
 p=\frac{4s}{d}+2, \quad \text{namely,}\quad
2s =  \frac {dp}2 - d, 
 \label{up0a}
\end{align}

\noi
where $s \in \N$.
Given  $K> \|Q\|_{L^2 (\R^d)}$, define the energy functional $E_K(u)$ by 
\begin{align}
E_K(u)
= \frac{K^p}{p}\|u\|^{p}_{L^p(\R^d)}-\frac{K^2}{2}\|u\|^2_{\dot H^{s}(\R^d)}.
\label{EE1}
\end{align}

\noi
which is the Euclidean counterpart of the energy functional $\EE_K(u)$
defined in \eqref{QQ1}.
Then, 
by rescaling, 
we can rewrite 
$\CC_K$  in \eqref{A2}
as 
\begin{align}
\CC_K=\sup_{\substack{\|u\|_{L^2(\R^d)} = 1\\ 
u = \P u}} E_K(u).
\label{up0b}
\end{align}

\subsection{Upper bound}
\label{SUBSEC:3.1}

In this subsection, we prove the following upper bound:
\begin{align}
\limsup_{N\to \infty}N^{-2s} \log Z_{K, N}\le \CC_K, 
\label{up0}
\end{align}

\noi
where
$Z_{K, N} = Z_{d, s, p, K, N}$
is as in \eqref{Z1} 
and $\CC_K$ is as in \eqref{up0b}.

We first perform a preliminary computation.
Let $\eps >0$ (to be chosen later).
Then,
by Cauchy's inequality, we have 
\begin{align}
\label{up2}
\begin{split}
\|Y_N+\Dr_N\|^2_{\dot H^{s}(\T^d)}
&=\|Y_N\|^2_{\dot H^{s}(\T^d)}+\|\Dr_N\|^2_{\dot H^{s}(\T^d)}+2\text{Re}\,
\jb{Y_N,\Dr_N}_{\dot H^{s}(\T^d)}\\
& \le \frac{1}{1-N^{-\eps}} \|\Dr_N\|^2_{\dot H^{s}(\T^d)}
+N^\eps \|Y_N\|^2_{\dot H^{s}(\T^d)}.
\end{split}
\end{align}

\noi
Then, from 
\eqref{up2}, 
\eqref{Y3}, and 
Bernstein's inequality, 
we have 
\begin{align}
\begin{split}
\E\big[\|\Dr_N\|^2_{\dot H^{s}(\T^d)}\big]
&\ge (1-N^{-\eps}) \E\big[\|Y_N+\Dr_N\|^2_{\dot H^{s}(\T^d)}\big]-cN^{d+\eps}\\
& \ge 
\E\big[
\|Y_N+\Dr_N\|^2_{\dot H^{s}(\T^d)} \cdot \ind_{\{\|Y_N+\Dr_N\|_{L^{2}(\T^d)}\le K\}}\big]\\
& \quad -K^{2}N^{2s-\eps}
\E\big[ \ind_{\{\|Y_N+\Dr_N\|_{L^{2}(\T^d)}\le K\}}\big]-cN^{d+\eps}\\
& \ge 
\E\big[
 \|Y_N+\Dr_N\|^2_{\dot H^{s}(\T^d)} \cdot \ind_{\{\|Y_N+\Dr_N\|_{L^{2}(\T^d)}\le K\}}
\big]
- o(N^{2s}), 
\end{split}
\label{up3}
\end{align}

\noi
as $N \to \infty$, 
where the last step follows from choosing $\eps > 0$ sufficiently small
such that $2s > d  + \eps$.
Hence, 
from Lemma \ref{LEM:var} with \eqref{var4} and \eqref{up3}, 
we have 
\begin{align}
\log  Z_{K, N}
& \le \log \E_{\mu}\bigg[
\exp\bigg( \frac 1 p\| \P_N u \|_{L^p(\T^d)}^p \cdot \ind_{\{ \| \P_N u \|_{L^2 (\mathbb{T}^d)}\leq K \}}  \bigg)\bigg]\notag\\
&  \le \sup_{\dr\in \mathbb{H}_a}\E\bigg[
\frac{1}{p}\|Y_N+\Dr_N\|^p_{L^p(\T^d)} \cdot 
\ind_{\{\|Y_N+\Dr_N\|_{L^2(\T^d)}\le K\}}-\frac{1}{2}\|\Dr_N \|_{\dot H^s(\T^d)}^2\bigg]\notag \\
& \le \sup_{\dr\in \mathbb{H}_a}\E\bigg[\bigg(
\frac{1}{p}\|Y_N+\Dr_N\|^p_{L^p(\T^d)}
-\frac{1}{2}\|Y_N+\Dr_N\|^2_{\dot H^{s}(\T^d)}\bigg) \notag \\
& \hphantom{XXXXX}
\times  \ind_{\{\|Y_N+\Dr_N\|_{L^2(\T^d)}\le K\}}\bigg]+o(N^{2s})
\label{up4}
\\
& \le \sup_{\substack{ \|u\|_{L^2(\T^d)}\le K\\u = \P_N u\\}}
\bigg(\frac{1}{p}\|u\|_{L^p(\T^d)}^p-\frac{1}{2}\|u\|^2_{\dot H^{s}(\T^d)}\bigg)\notag \\
& \hphantom{XXXXX}
 \times \sup_{\theta \in \mathbb{H}_a} \PP \Big(\{ \| Y_N +
\Theta_N \|_{L^2 (\mathbb{T}^d)} \leq K \}\Big) +o(N^{2s})\notag \\
& \le \sup_{\substack{  \|u\|_{L^2(\T^d)}\le 1\\u = \P_N u}}
\EE_K(u)
+o(N^{2s}), \notag
\end{align}

\noi
as $N \to \infty$, 
where $\EE_K(u)$ is as in \eqref{QQ1}.
Hence, 
by setting
\begin{align}
\CC_{K,N}
 = 
 \sup_{\substack{\|u\|_{L^2(\T^d)}\le 1\\u = \P_N u}}
\EE_K(u), 
\label{up6x}
\end{align}

\noi
the desired bound
\eqref{up0} follows 
once we prove
\begin{align}
\limsup_{N\to \infty}N^{-2s}\CC_{K,N}
: = 
\limsup_{N\to \infty}N^{-2s}  \sup_{\substack{\|u\|_{L^2(\T^d)}\le 1\\u = \P_N u}}
\EE_K(u)
\le \CC_K.
\label{up6}
\end{align}

Under the assumption $K>\|Q\|_{L^{2}(\R^d)}$, 
it follows from 
 Lemma \ref{LEM:sol2}
that $\CC_{K, N} > 0$. 
By noting that $\EE_K\big(\|u\|_{L^2(\T^d)}^{-1} u \big) 
\ge \EE_{K}(u)$ for 
$0 < \|u\|_{L^2(\T^d)} \le 1$, we can rewrite $\CC_{K, N}$ in \eqref{up6x} as 
\begin{align}
\CC_{K,N}=\sup_{\substack{\|u\|_{L^2(\T^d)}= 1\\u = \P_N u}}
\EE_{K}(u). 
\label{up7}
\end{align}

\noi
Moreover, 
it follows from 
 the compactness of the  set $\big\{u \in L^2(\T^d): u = \P_Nu, \, \|u\|_{L^{2}(\T^d)}=  1\big\}$
 that,  
for each $N \in \N$, 
there exists an optimizer $f_N$
for~\eqref{up7} (such that $f_N = \P_N f_N$).
In the following, when we view $f_N$ as a function on $\R^d$, we simply view it
as a periodic function: $f (x) = f (x + k)$, $k \in \Z^d$.

Fix small $\eps>0$. Let $\chi_{\eps}\in C_c^{\infty}(\R^d;[0,1])$ be a smooth bump function
compactly supported on $\big[-\tfrac 12 ,\tfrac 12 \big)^d\cong \T^d$
such that $\chi_{\eps}\equiv1$ on 
$\big[-\frac 12 + c_0 \eps ,\frac 12 - c_0 \eps \big)^d$
 for some small $c_0>0$
(to be chosen later) 
and
\begin{align}
\|\partial^{\alpha}\chi_{\eps}\|_{L^{\infty}(\R^d)}\les \eps^{-|\alpha|}.
\label{chi1}
\end{align}

\noi
for any multi-index $\al$.
As in the proof of \cite[Lemma 2.2]{OST}, 
we claim that, 
by translating $f_N$ (that does not affect its optimality for \eqref{up7}) and choosing $c_0$ sufficiently small (independent of small
$\eps > 0$ and $N\in \N$), we have
\begin{align}
\begin{split}
 \|\chi^2_{\eps}f_N\|_{L^{p}(\R^d)}^p = \|\chi^2_{\eps}f_N\|_{L^{p}(\T^d)}^p&\ge (1-C\eps)\|f_N\|_{L^{p}(\T^d)}^p,
\\
\|\chi^2_{\eps}f_N\|_{L^{2}(\R^d)}^2
= 
\|\chi^2_{\eps}f_N\|_{L^{2}(\T^d)}^2&\ge (1- C\eps)\|f_N\|_{L^{2}(\T^d)}^2.
\end{split}
\label{chi2}
\end{align}

\noi
Indeed, by writing 
$\big[-\frac 12 ,\frac 12 \big)^d$
as a union of  disjoint (modulo boundaries)
congruent cubes $R_j$ of side length $\sim \eps$
with centers $x_j$, $j \in I_0$ with $|I_0| \sim \eps^{-d}$, 
we first write
$\big \{x \in \big[-\frac 12 ,\frac 12 \big)^d; \chi_\eps(x) \ne 1\big\} \subset A:= \bigcup_{j \in I_1} R_j$
for some index set $I_1$ with $|I_1| \sim \eps^{-d+1}$.
Then, given $q = 2$ or $p$, we have
\begin{align*}
\sum_{j \in I_0} \| f_N \|_{L^q(x_j + A)}^q
= C_1  \eps^{-d+1} \|f_N\|_{L^q(\T^d)}^q
\end{align*}

\noi
for some $C_1 > 0$.
Then, 
we see that there exist at least $\frac 23 |I_0| \sim \eps^{-d}$ many $x_j$'s
such that 
\begin{align}
\| \chi_\eps^2 f_N \|_{L^q(x_j + A)}^q \le \| f_N \|_{L^q(x_j + A)}^q
 < \frac 3{|I_0|}
C_1 \eps^{-d+1} \|f_N\|_{L^q(\T^d)}^q
\sim 
\eps \|f_N\|_{L^p(\T^d)}^p,
\label{chi2a}
\end{align}

\noi
which implies that we can  translate $f_N$ by $x_j$
such that 
both of the bounds in \eqref{chi2} 
follow from~\eqref{chi2a} and the triangle inequality
(and there are $O(\eps^{-d})$-many choices for such $x_j$).

Let $\eta \in C_c^{\infty}(\R^d;[0,1])$ be a smooth radial bump function on $\R^d$ such that 
$\eta(\xi)=1$ for $|\xi|\le 1$
and $\eta(\xi)=0$ for $|\xi|>2$.
 Given $M>0$, let $\eta_M(\xi)=\eta(\frac{\xi}{M})$
 and set 
\begin{align}
\chi_{\eps,M}=\mathcal{F}^{-1}_{\R^d}(\eta_M)*\chi_{\eps}.
\label{chi4}
\end{align}

\noi
Since $\chi_\eps$ is a Schwartz function, 
there exists $M = M(\eps) \gg 1$ for each fixed small $\eps > 0$ such that 
\begin{align}
\| \dd^\al (\chi_{\eps}-\chi_{\eps,M})\|_{L^{1}(\R^d)\cap L^{\infty}(\R^d)}\ll \eps
\label{chi5}
\end{align}

\noi
for any multi-index $\al$
with $|\al|\le s$.
Moreover, proceeding as in \cite[(2.9)]{OST}, we have 
\begin{align}
\begin{split}
\|\chi_{\eps,M}(\,\cdot+m)\|_{L^{\infty}([-\frac 12,\frac 12)^d)}
&\le M^d\sup_{x\in [-\frac 12 ,\frac 12)^d}\bigg |\int_{\R^d}\chi_\eps(x+m-y)\mathcal{F}^{-1}_{\R^d}(\eta)(My)dy\bigg|
\\
& \les \frac{M^d}{\jb{Mm}^{2d+1} }\ll \frac{\eps}{\jb{m}^d}, 
\end{split}
\label{chi6}
\end{align}

\noi
uniformly in $m \in \Z^d\setminus \{0\}$.
A similar computation with \eqref{chi1} yields
\begin{align}
\|\partial^{\alpha}(\chi_{\eps,M}(\,\cdot+m))\|_{L^{\infty}([-\frac12,\frac 12)^d)}
 \les \frac{\eps^{1-|\alpha|}}{\jb{m}^d}.
\label{chi7}
\end{align}

\noi
We also note that, 
by proceeding as in  \cite[(2.10)]{OST} 
with 
\eqref{chi5} and \eqref{chi6}, 
we have 
\begin{align}
\|(\chi^2_{\eps,M} - \chi^2_{\eps})f_N\|^p_{L^{p}(\R^d)}
\ll \eps^p \|f_N\|^p_{L^{p}(\T^d)}.
\label{chi8}
\end{align}

Given $N \in \N$ and $M =  M(\eps)$ as above, 
 we introduce a function $g_{N,M}$ on $\R^d$ by setting
\begin{align}
g_{N,M} (x) = N^{- \frac d2} \chi_{\eps,M}^2 (N^{-1}x) f_N (N^{-1} x).
\label{GG1}
\end{align}

\noi
Then, by a change of variables with \eqref{up0a}, we have
\begin{align}
\begin{split}
\|g_{N,M}\|^2_{\dot H^s(\R^d)}
&=N^{-2s}\| \chi^2_{\eps,M} f_N\|^2_{\dot H^s(\R^d)},\\
\|g_{N,M}\|^{p}_{L^{p}(\R^d)}
&=N^{-2s}
\|\chi^2_{\eps,M}f_N\|^p_{L^{p}(\R^d)}.
\end{split}
\label{GG2}
\end{align}

\noi
Moreover, from 
\eqref{GG2}, \eqref{chi8},  
and 
\eqref{chi2}, we have 
\begin{align}
\|g_{N,M}\|^{p}_{L^{p}(\R^d)}
\ge  (1- C\eps) 
N^{-2s}\|f_N\|^p_{L^{p}(\T^d)}.
\label{GG3}
\end{align}

\noi
From \eqref{GG3} (with $p = 2$ and $s = 0$), \eqref{GG1}, and $\|f_N\|_{L^2(\T^d)} = 1$, we also have 
\begin{align}
\|g_{N,M}\|_{L^{2}(\R^d)}^2 = 1+ O(\eps).
\label{GG4}
\end{align}

Next, we establish a lower bound on $\|f_N\|_{\dot H^s(\T^d)}^2$
in terms of  $\|g_{N,M}\|^2_{\dot H^s(\R^d)}$,
where we crucially use the fact that $s \in \N$.
Given any multi-indices $\al_2$ and $\al_3$, 
it follows from 
 \eqref{chi6} and \eqref{chi5} 
 with 
$\| f_N \|_{L^2(\T^d)} = 1$
that 
\begin{align}
\begin{split}
\| &\dd^{\al_2}(\chi_{\eps,M}-\chi_{\eps}) \cdot \partial^{\alpha_3}f_N \|^2_{L^{2}(\R^d)}\\
&=\int_{[-\frac 12,\frac 12)^d}
|\dd^{\al_2}(\chi_{\eps,M}-\chi_{\eps})(x)|^2
| \partial^{\alpha_3}f_N(x)|^2 dx\\
& \quad 
+\sum_{m\in \Z^d\setminus \{0\}}\int_{[-\frac 12,\frac 12)^d}
|\dd^{\al_2}\chi_{\eps,M}(x+m)|^2 
|\partial^{\alpha_3}f_N(x)|^2dx\\
&\le \bigg(\| \dd^{\al_2}(\chi_{\eps,M}-\chi_{\eps})\|^2_{L^{\infty}(\R^d)}
+   \sum_{m\in \Z^d\setminus \{0\}}\frac {\eps^{2(1-|\al_2|)}}{\jb{m}^{2d}}
\bigg)
\|\dd^{\al_3} f_N \|_{L^2(\T^d)}^2\\
& \les \eps^{2(1-|\al_2|)}
N^{2|\al_3|}.
\end{split}
\label{GG6}
\end{align}

\noi
Then, 
from 
\eqref{Hs4},  
the  Leibniz rule, \eqref{GG6},  \eqref{chi5}, and  \eqref{chi1}, we have 
\begin{align}
\begin{split}
& \|  (\chi^2_{\eps,M}-\chi^2_{\eps}) f_N\|_{\dot H^{s}(\R^d)}^2
\le \sum_{|\alpha|=s}\binom{s}{\alpha}
\|\dd^{\alpha}((\chi^2_{\eps,M}-\chi^2_{\eps}) f_N) \|^2_{L^2(\R^d)}\\
&\quad \le \sum_{|\al_1|+ |\al_2|+|\al_3|=s}
C_{s, \al_1, \al_2, \al_3}
\|\dd^{\alpha_1}(\chi_{\eps,M}+\chi_{\eps})
\cdot
\dd^{\alpha_2}(\chi_{\eps,M}- \chi_{\eps}) 
\cdot \dd^{\alpha_3} f_N \|^2_{L^2(\R^d)}\\
&\quad \le \sum_{|\al_1|+ |\al_2|+|\al_3|=s}
C_{s, \al_1, \al_2, \al_3}'
\|\dd^{\alpha_1}(\chi_{\eps,M}+\chi_{\eps})\|_{L^\infty(\R^d)}^2
 \eps^{2(1-|\al_2|)}
N^{2|\al_3|}\\
&\quad \le \sum_{|\al_1|+ |\al_2|+|\al_3|=s}
C_{s, \al_1, \al_2, \al_3}''
 \eps^{2(1-|\al_1| - |\al_2|)}
N^{2|\al_3|}.
\end{split}
\label{GG6a}
\end{align}

\noi
Thus, 
by Cauchy's inequality, we have 
\begin{align}
\begin{split}
 \|  \chi^2_{\eps,M} f_N\|_{\dot H^{s}(\R^d)}^2
& \le   \Big(\|  \chi^2_{\eps} f_N\|_{\dot H^{s}(\R^d)}
+ C \eps N^{s}  + C\eps^{-s + 1} N^{s-1}\Big)^2\\
& \le  (1 + \eps ) \|  \chi^2_{\eps} f_N\|_{\dot H^{s}(\R^d)}^2
+ C \eps^{-1}\Big(\eps N^{s}  + C\eps^{-s + 1} N^{s-1}\Big)^2.
\end{split}
\label{GG5}
\end{align}

On the other hand, 
from \eqref{Hs4}, 
the Leibniz rule, the fact that $ \chi^2_\eps \in [0, 1]$ 
is supported on
$\big[-\tfrac 12 ,\tfrac 12 \big)^d\cong \T^d$, 
 \eqref{chi1},  and Bernstein's inequality
  with 
  $f_N = \P_N f_N$ and 
$\| f_N \|_{L^2(\T^d)} = 1$, 
we have
\begin{align}
\begin{split}
\| \chi^2_{\eps} f_N \|_{\dot H^{s}(\R^d)}
&\le \bigg(\sum_{|\alpha|=s}\binom{s}{\alpha}
\|\chi^2_{\eps} \cdot \dd^{\alpha}f_N \|^2_{L^2(\R^d)}
\bigg)^\frac 12 \\
&\quad
+\bigg(\sum_{|\alpha|=s}\binom{s}{\alpha}
\bigg\|\sum_{\be< \alpha}\binom{\alpha}{\be}\partial^{\alpha-\be}\chi^2_{\eps}
\cdot \partial^{\be}f_N
\bigg\|^2_{L^2(\R^d)} \bigg)^\frac 12 \\
& \le \|f_N\|_{\dot H^{s}(\T^d)}+C \eps^{-s}N^{s-1}.
\end{split}
\label{GG6b}
\end{align}

\noi
By squaring both sides
and applying Bernstein's inequality
  with 
$\| f_N \|_{L^2(\T^d)} = 1$, 
 we have
\begin{align}
\begin{split}
N^{-2s}\| \chi^2_{\eps} f \|_{\dot H^{s}(\R^d)}^2
& \le
N^{-2s} \|f_N\|_{\dot H^{s}(\T^d)}^2+
C\eps^{-s}N^{-1}
+ 
C \eps^{-2s}N^{-2}.
\end{split}
\label{GG7}
\end{align}

\noi
Hence, from 
\eqref{GG2}, \eqref{GG5}, 
and 
\eqref{GG7}, we obtain
\begin{align}
\begin{split}
\|g_{N,M}\|^2_{\dot H^s(\R^d)}
&=N^{-2s}\| \chi^2_{\eps,M} f_N\|^2_{\dot H^s(\R^d)}\\
&\le 
(1 + \eps )
N^{-2s}\| \chi^2_{\eps} f_N\|^2_{\dot H^s(\R^d)}
+ C\eps + C\eps^{-2s + 1} N^{-2}\\
&\le (1 + \eps ) N^{-2s}\|  f_N\|^2_{\dot H^s(\R^d)}
+ C\eps + C_\eps N^{-1}
\end{split}
\label{GG8}
\end{align}

\noi
for some constant $C_\eps > 0$, depending on $\eps > 0$
(which diverges as $\eps \to 0$).

Let $\EE_K(u)$ be as in 
\eqref{QQ1}.
Then, from  \eqref{GG3} and \eqref{GG8}, 
we have 
\begin{align}
\begin{split}
N^{-2s} \EE_K(f_N)
& = 
N^{-2s}\bigg[ \frac{K^p}{p} \|f_N \|^p_{L^p (\T^d)} - \frac{K^2}2
\|f_N \|^2_{\dot H^s(\mathbb{T}^d)} \bigg]\\
&\le \frac{K^p}{p}(1-C\eps)^{-1}\|g_{N,M}\|^p_{L^{p}(\R^d)}
- \frac{K^2}{2}(1+\eps)^{-1}\|g_{N,M}\|^2_{\dot H^s(\R^d)}\\
& \quad + C(K) \eps + C(K, \eps) N^{-1}.
\end{split}
\label{GH1}
\end{align}

Recall from \eqref{chi4} that 
 $\supp \F_{\R^d} (\chi_{\eps, M}) \subset B_{2M}$, 
 where
$B_{2M}$ is as in \eqref{ball1}.
Then, from \eqref{GG1}, we have 
\begin{align*}
\supp (\F_{\R^d}(g_{N,M})) \subset B_{1+\frac{2M}N}, 
\end{align*}

\noi
as $N \to \infty$, uniformly in small $\eps > 0$.
Moreover, note  from 
\eqref{GG1} with \eqref{chi4}
that 
 $\|g_{N,M}\|_{L^{2}(\R^d)}\neq 1$ in general.
In view of these observations, 
 we define  
a function   $h_{N,M}$ on $\R^d$ by 
\begin{align}
h_{N,M}(x)=\|g_{N,M}\|^{-1}_{L^2(\R^d)}
\Big(1+\frac{2M}{N}\Big)^{-\frac{d}{2}}g_{N,M}\bigg(\Big(1+\frac{2M}{N}\Big)^{-1}x\bigg)
\label{GH3}
\end{align}

\noi
such that 
\begin{align}
\| h_{N, M}\|_{L^2(\R^d)} = 1
\qquad \text{and}\qquad 
h_{N,M} = \P h_{N,M}.
\label{GH4}
\end{align}

\noi
By Bernstein's inequality with \eqref{GH4}, we have 
\begin{align}
\|h_{N,M}\|_{\dot H^s (\R^d)}, \|h_{N,M}\|_{L^{p}(\R^d)}
\les \|h_{N,M}\|_{L^{2}(\R^d)}= 1.	
\label{GH5}
\end{align}

\noi
Hence, using
\eqref{GH3},  \eqref{GG4}, 
 \eqref{GH5}, and \eqref{up0b}
 with \eqref{GH4}, 
we can bound 
the main contribution in~\eqref{GH1} by 
\begin{align}
\begin{split}
&\frac{K^p}{p}(1-C\eps)^{-1}\|g_{N,M}\|^p_{L^{p}(\R^d)}
-\frac{K^2}{2}(1+\eps)^{-1}\| g_{N,M}\|^2_{\dot H^{s}(\R^d)}\\
&=\frac{K^p}{p}(1-C\eps)^{-1}\|g_{N,M}\|^{p}_{L^{2}(\R^d)}
\bigg(1+\frac{2M}{N}\bigg)^{-2s}\|h_{N,M}\|^p_{L^{p}(\R^d)}\\
&\quad
-\frac{K^2}{2}(1+\eps)^{-1}
\|g_{N,M}\|^{2}_{L^{2}(\R^d)}
\bigg(1+\frac{2M}{N}\bigg)^{-2s}\| h_{N,M}\|^2_{\dot H^{s}(\R^d)}\\
& \le E_K(h_{N, M})
+  C_{K, \eps}(M, N)
+ C_K(\eps)\\
& \le \CC_K 
+  C_{K, \eps}(M, N)
+ C_K(\eps), 
\end{split}
\label{GH7}
\end{align}

\noi
where $E_K(u)$ is as in \eqref{EE1}, 
and 
$C_{K, \eps}(M, N)$
and
$C_K(\eps)$ satisfy 
\begin{align}
\lim_{N \to \infty} C_{K, \eps}(M, N) = 0
\qquad \text{and} \qquad \lim_{\eps \to 0} C_K(\eps) = 0.
\label{GH8}
\end{align}

\noi
Here, the first convergence in \eqref{GH8}
follows from the fact that 
 $M = M(\eps)$ is independent of $N \in \N$
 (in particular $1 + \frac{2M}{N} \to 1$ as $N \to \infty$)
 for each fixed $\eps > 0$ 
(used in \eqref{chi5}, \eqref{chi6}, and~\eqref{chi7}).

Therefore, recalling that 
$f_N$ is an optimizer for \eqref{up7}, 
it follows from 
\eqref{up7}, 
 \eqref{GH1},  and \eqref{GH7} 
 that 
\begin{align*}
\limsup_{N\to \infty} N^{-2s}\CC_{K,N}
 = \limsup_{N\to \infty}
 N^{-2s}
  \EE_K(f_N)
\le 
\CC_K + 
C(K) \eps+ 
C_K(\eps).
\end{align*}

\noi
Finally, by taking $\eps$ to $0$
and applying \eqref{GH8}, 
we obtain
\eqref{up6}
(and hence \eqref{up0}).

\begin{remark}\label{REM:up1}\rm

In the argument presented above, 
we crucially relied on the fact that $s$ is an integer
in applying \eqref{Hs4} (with the fact that
$\dd^\al$ is a local operator)
and also the Leibniz rule;
see \eqref{GG6}, \eqref{GG6a}, and \eqref{GG6b}.
While we expect that the upper bound \eqref{up0}
also holds for non-integer $s > \frac d2$
(by making use of fractional calculus; see
\cite{BOZ} and the references therein), 
we do not pursue this issue
in the present paper.

\end{remark}

\subsection{Lower bound}
\label{SUBSEC:3.2}

Next,  we prove the following lower bound:
\begin{align}
\liminf_{N\to \infty}N^{-2s} \log Z_{K, N}\ge \CC_K
\label{low0}
\end{align}

\noi
for $ s > \frac d2$ and $p$ satisfying \eqref{up0a}, 
where
$Z_{K, N} = Z_{d, s, p, K, N}$
is as in \eqref{Z1} 
and $\CC_K$ is as in~\eqref{A2}.
As pointed out in Remark \ref{REM:1}, 
we do not need to assume that $s$ is an integer
for the lower bound~\eqref{low0}.
By following the ideas developed in \cite{OST}, 
we prove \eqref{low0}
by constructing a suitable drift 
in~\eqref{var2} of Lemma \ref{LEM:var}
for each $N \in \N$.

Set 
\begin{align}
 \mathcal{A}
= \big\{
f \in \S(\R^d):  
\|f\|_{L^2(\R^d)} = 1, \,
\supp \F_\R^d(f) \subset B_1, \, 
\F_{\R^d}(f) (0) = 0\big\}.
\label{AA1}
\end{align}

\noi
Then, given any $s \ge 0$ and $p \ge 2$, 
it follows from 
Bernstein's inequality that 
\begin{align}
\| f\|_{\dot H^s(\R^d)}, \|f\|_{L^p(\R^d)} \les \| f\|_{L^2(\R^d)} = 1
\label{AA2}
\end{align}

\noi
for any $f \in \mathcal{A}$.
Fix
 $f \in \mathcal{A}$.
Given $N \in \N$, 
 a function $F_N$  on $\T^d$ by 
\begin{align}
F_N (x) =  N^{-\frac d 2} \sum_{\substack{|n| \le N}} \F_{\R^d}(f)(N^{-1} n) e^{2\pi i n \cdot x}.
\label{low1} 
\end{align}

\noi
Note that $\ft F_N(0) = 0$ and that we have $F_N = f_{N^{-1}}^\per$, 
where the latter is as in 
\eqref{per3}.
In particular, Lemma \ref{LEM:sol1} holds; 
with \eqref{up0a}, we have 
\begin{align}
\| F_N  \|^2_{\dot{H}^s (\mathbb{T}^d)}
&=  N^{2s} \| f \|^2_{\dot H^{s} (\R^d)} + o(N^{2s}), 
\label{low2}\\
\| F_N  \|_{L^p (\mathbb{T}^d)}^p &=
N^{\frac {dp}2 - d } \| f \|^p_{L^p (\R^d)} + o (N^{\frac {dp}2 - d}).
  \label{low3}
\end{align}

Given $N \in \N$, we define  a drift $\dr_N$ by setting
\begin{align}
\dr_N (t) 
& = -D^{s}\frac{d}{dt} \ze_N(t) + \alpha D^s F_N , 
\label{low2a}
\end{align}

\noi
where 
$\ze_N$ is as in Lemma \ref{LEM:Z}
and 
$\alpha = \al(K, N)> 0$ is a constant given by 
\begin{align}
\alpha=K- N^{-\be}
\label{low4}
\end{align}

\noi
for some  $0<\be < \min\big(\frac s2 - \frac d4, \frac d4 \big)$.
Then, we  set 
\begin{align}\label{low5}
\Dr_{N}=
\int_0^1 D^{- s} \theta_N (t) dt
= 
-\ze_N
+ 
\alpha F_N, 
\end{align}

\noi
where we used the short-hand notation: $\ze_N = \ze_N(1)$.
Note that
from Lemma \ref{LEM:Z} and \eqref{low1}, we have 
$\Dr_N = \P_{N}\Dr_{N}$.

From  \eqref{Z1}, 
we have 
\begin{align}
  Z_{K, N}
\ge \E_{\mu}\bigg[ 
\exp\bigg( \frac 1 p \| \P_N u\|^p_{L^p(\T^d)} \cdot 
 \ind_{\{ \| \P_{N} u \|_{L^2 (\mathbb{T}^d)}\leq K \}}\bigg)\bigg]-1.
\label{low5a}
\end{align}

\noi
Thus, we see that  \eqref{low0} follows once we prove 
\begin{align}
\liminf_{N\to \infty}
N^{-2s} 
\log \E_{\mu}\bigg[ 
\exp\bigg( \frac 1 p \| \P_N u\|^p_{L^p(\T^d)} \cdot 
 \ind_{\{ \| \P_{N} u \|_{L^2 (\mathbb{T}^d)}\leq K \}}\bigg)\bigg] \ge \CC_K.
\label{low5b}
\end{align}

\noi
From 
Lemma \ref{LEM:var}
and \eqref{low5}, we have 
\begin{align}
\begin{split}
& \log \E_{\mu}\bigg[ 
\exp\bigg( \frac 1 p \| \P_N u\|^p_{L^p(\T^d)} \cdot 
 \ind_{\{ \| \P_{N} u \|_{L^2 (\mathbb{T}^d)}\leq K \}}\bigg)\bigg] \\
& \quad = \sup_{\theta\in \mathbb{H}_a}\E \bigg[
\frac{1}{p}\|Y_N+\Dr_N\|^p_{L^p(\T^d)} \cdot \ind_{\{\|Y_N+\Dr_N\|_{L^2(\T^d)}\le K\}}
-\frac{1}{2}\int_{0}^{1}\|\theta_N(t)\|^2_{L_x^2(\T^d)}dt\bigg]\\
& \quad  \ge \E\bigg[\frac{1}{p}\|Y_N-\ze_N+\al F_N\|^p_{L^p(\T^d)}
 \cdot \ind_{\{\|Y_N-\ze_N+\al F_N\|_{L^2(\T^d)}\le K\}}\\
& \quad  \phantom{XXX}
-\frac{1}{2}\int_{0}^{1}\|- \dt \ze_N(t)+\alpha F_N\|^2_{\dot H^s (\T^d) }dt\bigg]\\
& \quad 
=\frac{\alpha^p}{p}\|F_N\|^p_{L^p(\T^d)}-\frac{\alpha^2}{2}\|F_N\|^2_{\dot H^{s}(\T^d)}
+\sum_{j=1}^3B_j,
\end{split}
\label{low6}
\end{align}

\noi
where $B_j$, $j = 1, 2, 3$, are given by 
\begin{align}
\begin{split}
 B_1& =-\frac{1}{p}
\E\bigg[\Big(\|\alpha F_N\|^p_{L^{p}(\T^d)}-\|Y_N-\ze_N+\alpha F_N\|^p_{L^p(\T^d)}\Big)\\
& \hphantom{XXXXX}
\times  \ind_{\{\|Y_N-\ze_N+\alpha F_N\|_{L^2(\T^d)}\le K\}}\bigg],\\
B_2 & =-\frac{\alpha^p}{p}\E\bigg[\|F_N\|^p_{L^p(\T^d)}
\cdot \ind_{\{\|Y_N-\ze_N+\alpha F_N\|_{L^2(\T^d)}>K\}} \bigg],\\\
B_3 & =-\frac{1}{2}\E\bigg[\int_{0}^1 
\|\dt \ze_N(t)\| ^2_{\dot H^s(\T^d)}-2 
\jb{ \dt \ze_N(t), \alpha F_N }_{\dot H^{s}(\T^d)}dt\bigg].
\end{split}
\label{low6a}
\end{align}

In the following, we first prove
\begin{align}
\sum_{j=0}^{3}|B_j|=o(N^{2s}).
\label{low7}
\end{align}

\noi
From the  mean value theorem, we have 
\begin{align*}
& \| \alpha F_N \|^p_{L^p (\mathbb{T}^d)} - \| Y_N - \ze_N  +
\alpha F_N \|^p_{L^p (\mathbb{T}^d)}\\
& \quad \lesssim \int_{\mathbb{T}^d} | Y_N - \ze_N  |^p + | Y_N - \ze_N| | \alpha F_N |^{p-1} d x.
\end{align*}

\noi
Fix  small $\eps > 0$.
Then, it follows from \eqref{low6a}, 
Lemma \ref{LEM:Z},  and 
\eqref{low3} that 
\begin{align}
\begin{split}
|B_1| & \lesssim 
\big\{\max (N^{- s + \frac{d}{2}}, N^{- \frac{d}{2}+\eps})\big\}^{\frac{p}{2}} 
+ \big\{\max (N^{- s + \frac{d}{2}}, N^{- \frac{d}{2}+\eps})\big\}^{\frac{1}{2}}
\| F_N \|^{p-1}_{L^{p}(\mathbb{T}^d)}  \\
&=o (N^{\frac {dp}2 - d}) = o (N^{2s}).
\end{split}
\label{low9}
\end{align}

\noi
 By \eqref{low3}, 
Chebyshev's inequality,  
Lemma \ref{LEM:Z}, and \eqref{low4}, we have 
\begin{align}
\begin{split}
|B_2| 
& \lesssim \| F_N \|^p_{L^p(\mathbb{T}^d)} \cdot
\PP \Big(
 \| Y_N - \ze_N \|_{L^2 (\mathbb{T})} > K - \alpha \| F_N \|_{L^2 (\mathbb{T}^d)} \Big)\\
& \les N^{\frac {dp}2 - d} \frac{\mathbb{E} \big[\|Y_N - \ze_N \|_{L^2 (\mathbb{T}^d)}^2\big]}
{\big(K - \alpha \| F_N \|_{L^2 (\mathbb{T}^d)} \big)^2}\\
& \les N^{\frac {dp}2 - d} N^{2\be} \max (N^{-s+\frac{d}{2}}, N^{-\frac{d}{2} +\eps})
= o (N^{\frac {dp}2 - d})
=o(N^{2s}), 
\end{split}
\label{low10}
\end{align}

\noi
 since 
it follows from \eqref{low2} and \eqref{AA2} with $s = 0$
that 
$\| F_N \|_{L^2 (\mathbb{T}^d)} = 1 + o(1)$.
Noting that  
$\jb{ \dt \ze_N(t), \alpha F_N }_{\dot H^{s}(\T^d)}$
is a mean-zero  Gaussian random variable for each $0 \le t \le 1$, 
it follows from 
Lemma \ref{LEM:Z}  that 
\begin{align}
|B_3|  =\frac{1}{2}\E\bigg[\int_{0}^1 
\|\dt \ze_N(t)\| ^2_{\dot H^s(\T^d)}dt \bigg]
\les 
\max
(N^{\frac{3 }{2}d - s}, N^{\frac{d}{2} +\eps })
= o (N^{2s}), 
\label{low11}
\end{align}

\noi
since $ s > \frac d2$.
Hence, \eqref{low7} follows
from \eqref{low9}, \eqref{low10}, and \eqref{low11}.

Therefore, 
from
 \eqref{low6}, \eqref{low7}, \eqref{low2}, \eqref{low3}, and \eqref{low4}, 
we obtain
\begin{align}
\begin{split}
& N^{-2s} 
 \log \E_{\mu}\bigg[ 
\exp\bigg( \frac 1 p \| \P_N u\|^p_{L^p(\T^d)} \cdot 
 \ind_{\{ \| \P_{N} u \|_{L^2 (\mathbb{T}^d)}\leq K \}}\bigg)\bigg]\\
& \quad 
\ge \frac{(K-N^{-\be})^p}{p}\|f\|^p_{L^p(\R^d)}-\frac{(K-N^{-\be})^2}{2}\|f\|^2_{\dot H^{s}(\R^d)}
+ o(1)\\
& \quad = E_K(f) + o(1), 
\end{split}
\label{low12}
\end{align}

\noi
as $N \to \infty$, 
where $E_K(f)$ is as in \eqref{EE1}.

Note that $\mathcal{A}$ defined in \eqref{AA1} 
is dense in the class:
\begin{align}
 \mathcal{A}_0 = \{ f: \|f\|_{L^2(\R^d)} = 1\}\cap 
 \{f : \supp \F_\R^d(f) \subset B_1\}.
\label{low13}
\end{align}

\noi
Moreover, 
by Bernstein's inequality \eqref{AA2}, 
we see
that $\mathcal{A}$ is dense in 
$\mathcal{A}_0 \cap \dot H^s(\R^d)\cap L^p(\R^d)$.
Then, given 
small $\eps > 0$, 
it follows from \eqref{up0b}
and the observation above that 
we can choose $f \in \mathcal{A}$
such that 
\begin{align}
E_K(f)
\ge  \sup_{\substack{\|u\|_{L^2(\R^d)} = 1\\ 
u = \P u}} E_K(u) - \eps 
= \CC_K - \eps.
\label{low14}
\end{align}

\noi
Hence, by putting 
\eqref{low12}
and \eqref{low14} together, we have
\begin{align*}
\liminf_{N\to \infty}
N^{-2s} 
\log \E_{\mu}\bigg[ 
\exp\bigg( \frac 1 p \| \P_N u\|^p_{L^p(\T^d)} \cdot 
 \ind_{\{ \| \P_{N} u \|_{L^2 (\mathbb{T}^d)}\leq K \}}\bigg)\bigg] \ge \CC_K-\eps
\end{align*}

\noi
Since the choice of $\eps > 0$ 
is arbitrary, 
we therefore conclude 
\eqref{low5b} (and hence \eqref{low0}).
This conclude the proof of Theorem \ref{THM:1}.

\begin{remark}\label{REM:non1}\rm

In the following, we briefly discuss how to obtain
\eqref{ZA2}
in Remark \ref{REM:mean}.
With  $W$  in \eqref{W1}, 
define $\wt Y$ by 
\begin{equation*}
\wt Y (t) = \jb{\nb}^{- s} W (t).
\end{equation*}

\noi
Note that we have
$\Law (\wt Y (1)) = \wt \mu$,
where $\wt \mu$ is 
 the massive fractional Gaussian free field
in~\eqref{gauss2}.
Given $N \in \N$, set
$\wt Y_N = \P_{N} \wt Y (1)$.
Given a drift $\dr \in \Ha$, 
we also define
 $\wt \Dr $ by 
\begin{equation*}
\wt \Theta = \int_0^1 \jb{\nb}^{- s} \theta (t) dt
\end{equation*}

\noi
and $\wt \Dr_N = \P_N \wt \Dr(1) $.
We note that, 
instead of \eqref{var4}, 
we have 
\begin{equation}
\| \wt \Theta \|^2_{H^s (\mathbb{T}^d)}
\leq \int_0^1 \| \theta (t) \|^2_{L^2 (\mathbb{T}^d)} d t
\label{ZA5} 
\end{equation}

\noi
for any $\ta \in \Ha$.
With these notations, 
 the variational formula (an analogue of Lemma \ref{LEM:var})  reads as 
\begin{align}
\log\E \big[e^{F (\wt Y_N)}\big] = \sup_{\theta \in\Ha} 
\E \bigg[ F (\wt Y_N + \wt \Theta_N) - \frac{1}{2} \int_0^1\| \theta (t) \|^2_{L^2 (\mathbb{T}^d)} d t \bigg].
\label{ZA6} 
\end{align}

\noi
Then, 
proceeding as in \eqref{up4}
with 
\eqref{ZA6} and \eqref{ZA5}, 
we have 
\begin{align*}
\log  \wt Z_{K, N}
&  \le \sup_{\dr\in \mathbb{H}_a}\E\bigg[
\frac{1}{p}\|\wt Y_N+\wt \Dr_N\|^p_{L^p(\T^d)} \cdot 
\ind_{\{\|\wt Y_N+\wt \Dr_N\|_{L^2(\T^d)}\le K\}}-\frac{1}{2}\|\wt \Dr_N \|_{ H^s(\T^d)}^2\bigg]\notag \\
&  \le \sup_{\dr\in \mathbb{H}_a}\E\bigg[
\frac{1}{p}\|\wt Y_N+\wt \Dr_N\|^p_{L^p(\T^d)} \cdot 
\ind_{\{\|\wt Y_N+\wt \Dr_N\|_{L^2(\T^d)}\le K\}}-\frac{1}{2}\|\wt \Dr_N \|_{\dot  H^s(\T^d)}^2\bigg], 
\end{align*}

\noi
where the second step follows
from the trivial bound: $\|f\|_{\dot H^s(\T^d)}
\le \|f\|_{H^s(\T^d)}$.
In particular, 
we have
\begin{align*}
\log  \wt Z_{K, N} \le \text{RHS of } \eqref{up4}.
\end{align*}

\noi
Hence, 
 by repeating the argument in Subsection \ref{SUBSEC:3.1}, 
we obtain the upper bound:
\begin{align}
\limsup_{N\to \infty}N^{-2s} \log \wt Z_{K, N}\le \CC_K, 
\label{ZA7}
\end{align}

\noi
where $\CC_K$ is as in \eqref{A2}.

Next, we turn our attention to the lower bound.
Let $F_N$ be as in 
\eqref{low1}.
Then, 
from the mean value theorem, 
\eqref{low3} (with $p = 2$), 
and the fact that $\| f\|_{L^2(\R^d)} = 1$, 
we have
\begin{align*}
\Big|\|F_N\|^2_{ H^{s}(\T^d)}
- \|F_N\|^2_{ \dot H^{s}(\T^d)}\Big|
& \les \sum_{|n|\le N} |n|^{2(s-1)}|\ft F _N(n)|^2\\
& \le  N^{2(s-1)} \Big(\| f \|^2_{L^2 (\R^d)}+o(1)\Big)
= o(N^{2s}), 
\end{align*}

\noi
as $N \to \infty$.
Thus, instead of \eqref{low2}, we have 
\begin{align}
\| F_N  \|^2_{H^s (\mathbb{T}^d)}
&=  N^{2s} \| f \|^2_{\dot H^{s} (\R^d)} + o(N^{2s}).
\label{ZA8}
\end{align}

\noi
Then, by proceeding as in 
\eqref{low6} with \eqref{ZA8}, we have 
\begin{align*}
& \log \E_{\wt \mu}\bigg[ 
\exp\bigg( \frac 1 p \| \P_N u\|^p_{L^p(\T^d)} \cdot 
 \ind_{\{ \| \P_{N} u \|_{L^2 (\mathbb{T}^d)}\leq K \}}\bigg)\bigg] \\
& \quad 
\ge \frac{\alpha^p}{p}\|F_N\|^p_{L^p(\T^d)}-\frac{\alpha^2}{2}\|F_N\|^2_{\dot H^{s}(\T^d)}
+\sum_{j = 1}^3 \wt B_j  + o(N^{2s}).
\end{align*}

\noi
where $\wt B_j$, $j = 1, 2, 3$, are 
 given by 
\begin{align*}
\wt  B_1& =-\frac{1}{p}
\E\bigg[\Big(\|\alpha F_N\|^p_{L^{p}(\T^d)}-\|\wt Y_N-\wt \ze_N+\alpha F_N\|^p_{L^p(\T^d)}\Big)\\
& \hphantom{XXXXX}
\times  \ind_{\{\|\wt Y_N-\wt \ze_N+\alpha F_N\|_{L^2(\T^d)}\le K\}}\bigg],\\
\wt B_2 & =-\frac{\alpha^p}{p}\E\bigg[\|F_N\|^p_{L^p(\T^d)}
\cdot \ind_{\{\|\wt Y_N-\wt \ze_N+\alpha F_N\|_{L^2(\T^d)}>K\}} \bigg],\\\
\wt B_3 & =-\frac{1}{2}\E\bigg[\int_{0}^1 
\|\dt \wt \ze_N(t)\| ^2_{ H^s(\T^d)}-2 
\jb{ \dt \wt \ze_N(t), \alpha F_N }_{H^{s}(\T^d)}dt\bigg].
\end{align*}

\noi
Here, $\wt \ze_N$ is an approximation to $\wt Y_N$, 
an analogue of $\ze_N$
in 
 Lemma \ref{LEM:Z}
adapted to the current non-homogeneous setting.
Namely, 
 given any $\eps > 0$, 
 we have 
\begin{align}
\begin{split}
 \mathbb{E} \Big[\|\wt \ze_N (1) - \wt Y_N \|^p_{L^p (\mathbb{T}^d)}\Big] 
 & \lesssim 
\big\{\max
(N^{- s + \frac{d}{2}}, N^{- \frac{d}{2} +\eps})\big\}^{\frac{p}{2}}, \\
 \mathbb{E} \bigg[ \int_0^1 \Big\|  \frac{d}{dt} \wt \ze_N (t) \Big\|_{H^s (\mathbb{T}^d)}^2 dt \bigg] 
& \lesssim \max
(N^{\frac{3 }{2}d - s}, N^{\frac{d}{2} +\eps }).
\end{split}
\label{ZA11}
\end{align}

\noi
for any $N \gg 1$ and  finite $ p \geq 1$.
The construction of such $\wt \ze_N$, satisfying 
the bounds  \eqref{ZA11}, follows from a straightforward
modification of the proof of Lemma \ref{LEM:Z}
and thus we omit details.
Then, 
 by repeating the argument in Subsection \ref{SUBSEC:3.2}, 
 we 
 obtain the lower bound:
\begin{align}
\liminf_{N\to \infty}N^{-2s} \log \wt Z_{K, N}\ge \CC_K.
\label{ZZ1}
\end{align}

Therefore,
 from \eqref{ZA7} and \eqref{ZZ1}, 
 we obtain \eqref{ZA2}.

\end{remark}

\section{Supercritical case}\label{SEC:4}

In this section, we present a proof of Theorem \ref{THM:2}
on the $L^2$-supercritical case:
\begin{align}
 p>\frac{4s}{d}+2, \quad \text{namely,}
 \quad 2s < \frac {dp}2 - d, 
 \label{sp1}
\end{align}

\noi
where $s > \frac d2$.

\subsection{Upper bound}
In this subsection, we  establish the following upper bound:
\begin{align}
\limsup_{N \to \infty}
N^{-\frac{dp}{2}+d}\log Z_{K, N}\le C_B \frac{K^{p}}{p}.
\label{sp0a}
\end{align}

By arguing as in \cite[Lemma 2.2]{OST},  we obtain
\begin{align}\label{sp1a}
\|\P_N f \|^p_{L^{p}(\T^d)}\le C_B N^{\frac{dp}{2}-d}(1+o(1))\|f\|^{p}_{L^{2}(\T^d)},
\end{align}

\noi
where $C_B$ is the optimal constant in Bernstein's inequality  \eqref{A4}.
Then, 
proceeding as in~\eqref{up4}
with a change of variable $V_N = Y_N + \Dr _N$
and \eqref{sp1a}, 
we have 
\begin{align*}
\log Z_{K, N} 
&  \le \sup_{\dr\in \mathbb{H}_a}\E\bigg[
\frac{1}{p}\|Y_N+\Dr_N\|^p_{L^p(\T^d)} \cdot 
\ind_{\{\|Y_N+\Dr_N\|_{L^2(\T^d)}\le K\}}\bigg]\notag \\
&  \le \sup_{\substack{\|V_N \|_{L^2(\T^d)} \le K\\V_N = \P_N V_N }}
\frac{1}{p}\|V_N \|^p_{L^p(\T^d)} \cdot 
\ind_{\{\|V_N \|_{L^2(\T^d)}\le K\}}\\
& \le  C_B \frac{K^p}{p} N^{\frac{dp}{2}-d}
+ o(N^{\frac{dp}{2}-d}), 
\end{align*}

\noi
as $N \to \infty$.
This proves the upper bound for 
\eqref{sp0a}.

\subsection{Lower bound}

Next, we establish the following lower bound:
\begin{align}
\liminf_{N \to \infty}
N^{-\frac{dp}{2}+d}\log Z_{K, N}\ge C_B \frac{K^{p}}{p}.
\label{sp1b}
\end{align}

Let $\mathcal{A}$ be as in \eqref{AA1}.
Given $\eps > 0$, 
it follows from the density of $\mathcal{A}$
in $\mathcal{A}_0$ defined in \eqref{low13}
that there exists $f \in \mathcal{A}$ such that 
\begin{align}
\|f\|^p_{L^{p}(\R^d)}\ge C_B-\eps, 
\label{sp2}
\end{align}

\noi
where $C_B$ is
the optimal constant in  Bernstein's inequality \eqref{A4}.
Proceeding as in Subsection~\ref{SUBSEC:3.2}, 
we define 
$F_N$ as in \eqref{low1} 
and  the drift $\dr_N$ as in \eqref{low2a}.
In this case, instead of \eqref{low7}, 
it follows from \eqref{low9}, \eqref{low10}, and \eqref{low11}
with \eqref{sp1}
that 
\begin{align}
\sum_{j=0}^{3}|B_j|=o(N^{\frac {dp}2 - d}).
\label{sp3a}
\end{align}

\noi
Then, by combining 
\eqref{low5a}, 
\eqref{low6}, 
\eqref{low2},  \eqref{low3}, and \eqref{sp3a}
with \eqref{sp1} and \eqref{AA2}, we have 
\begin{align}
\begin{split}
\log (Z_{K, N} + 1)
&\ge \frac{\alpha^p}{p}
N^{\frac{dp}{2}-d}
\|f\|^p_{L^{p}(\R^d)} 
+o(N^{\frac{dp}{2}-d})\\
& \quad -\frac{\alpha^2}{2}N^{2s}\|f\|^2_{\dot H^s(\R^d)}+o(N^{2s})
 +\sum_{j = 1}^{3}B_j\\
&=\frac{K^p}{p}
N^{\frac{dp}{2}-d}
\|f\|^p_{L^{p}(\R^d)}
+o(N^{\frac{dp}{2}-d}),
\end{split}
\label{sp3}
\end{align}

\noi
as $N \to \infty$, 
where the last step follows from  \eqref{low4}.
Therefore, from \eqref{sp3} and \eqref{sp2}, 
we obtain 
\begin{align*}
\liminf_{N \to \infty}
N^{-\frac{dp}{2}+d}\log Z_{K, N}\ge 
(C_B- \eps) \frac{K^{p}}{p} - \eps.
\end{align*}

\noi
Since the choice of $\eps > 0$ is arbitrary, 
we therefore conclude 
\eqref{sp1b}.
This conclude the proof of Theorem \ref{THM:2}.

\section{Declarations}

\noi
{\bf Funding.}
D.G., G.L, and T.O.~were supported by the European Research Council
(grant no.~864138 ``SingStochDispDyn'').
G.L.~was also supported by the EPSRC New Investigator Award (grant no. EP/S033157/1).
T.O.~was also supported by the EPSRC Mathematical Sciences Small
Grant (grant no. EP/Y033507/1).
Y.W. was supported by  the EPSRC New Investigator
Award (grant no.~EP/V003178/1)
and by
 the EPSRC Mathematical Sciences Small Grant
 (grant no.~UKRI1116).

\medskip

\noi
{\bf Competing interests.}
The authors have no competing interests to declare that are relevant to the content of this article.

\medskip

\noi
{\bf Data availability statement.}
This manuscript has no associated data.




%
%

\end{document}